\documentclass[11pt]{amsart}
\usepackage{mathrsfs} 
\usepackage[margin=1.25in]{geometry}
\usepackage{ifthen}
\usepackage{graphicx}
\usepackage{epsfig}
\usepackage{dsfont, amssymb}
\usepackage{amsfonts,dsfont,amssymb,amsthm,stmaryrd,bbm}
\usepackage{amsmath}

\usepackage{xcolor}
\usepackage{hyperref}

\usepackage[toc,page]{appendix} 
\usepackage{hyperref,mathtools}
\hypersetup{colorlinks=true,pdftitle=""
}

\let\originallesssim\lesssim
\let\originalgtrsim\gtrsim

\DeclareRobustCommand{\lesssim}{%
  \mathrel{\mathpalette\lowersim\originallesssim}%
}
\DeclareRobustCommand{\gtrsim}{%
  \mathrel{\mathpalette\lowersim\originalgtrsim}%
}

\makeatletter
\newcommand{\lowersim}[2]{%
  \sbox\z@{$#1<$}%
  \raisebox{-\dimexpr\height-\ht\z@}{$\m@th#1#2$}%
}
\makeatother

\newtheorem{thm}{Theorem}[section]

\newtheorem{remark}[thm]{Remark}
\newtheorem{lem}[thm]{Lemma}
\newtheorem{prop}[thm]{Proposition}
\newtheorem{coro}[thm]{Corollary}

\newcommand\independent{\protect\mathpalette{\protect\independent}{\perp}} 
\def\independent#1#2{\mathrel{\rlap{$#1#2$}\mkern2mu{#1#2}}}

\def\k{{\kappa}}

\def\Ent{{\rm Ent}}

\def\ep{\varepsilon}
\def\phi{\varphi}

\def\bee{\begin{eqnarray*}}
\def\ene{\end{eqnarray*}}
\usepackage{mathtools}

\begin{document}
 \author{}

\title{Exponential inequalities in probability spaces revisited}

\author{Ali Barki, Sergey G. Bobkov, Esther Bou Dagher and Cyril Roberto}

\thanks{Research of S.G.B. was partially supported by the NSF
grant DMS-2154001. E.B.D. was supported by the LMS Early Career Fellowship 
(reference: ECF-2022-02) and the EPSRC Maths Research Associates 2021 ICL 
(reference: EP/W522673/1). The first and last author are supported by the Labex MME-DII 
funded by ANR, reference ANR-11-LBX-0023-01 and the fondation Simone et Cino Del Duca, France.
This research has been conducted within the FP2M federation (CNRS FR 2036).}

\address{MODAL'X, UPL, Univ. Paris Nanterre, CNRS, F92000 Nanterre France}

\address{School of Mathematics, University of Minnesota, Minneapolis, MN, USA}

\address{Department of Mathematics, Imperial College London,
180 Queen’s Gate, SW7 2AZ, London, \phantom{...} United Kingdom}

\email{ali.barki@ens-rennes.fr, bobko001@umn.edu, 
esther.bou-dagher17@imperial.ac.uk, croberto@math.cnrs.fr}
%

\date{\today}

\maketitle
 
\begin{abstract}
We revisit several results on exponential integrability in probability 
spaces and derive some new ones. In particular, we give a quantitative
form of recent results by Cianchi-Musil and Pick in the framework 
of Moser-Trudinger-type inequalities, and recover Ivanisvili-Russell's 
inequality for the Gaussian measure. One key ingredient is the use of 
a dual argument, which is new in this context, that we also 
implement in the discrete setting of the Poisson measure on integers.
\end{abstract}

\section{Introduction}

The aim of this paper is to develop a number of upper bounds
on exponential moments of functions on the Euclidean and abstract probability 
spaces under certain smoothness-type conditions. For the first motivating example, 
suppose that we are given a Borel probability measure $\mu$ on 
$\mathbb{R}^n$ satisfying a logarithmic Sobolev inequality
\begin{equation} \label{eq:log-sobolev}
\int f e^f d\mu \leq c \int |\nabla f|^2 e^f d\mu \quad 
\mbox{for all smooth} \ f \ \mbox{on} \ \mathbb{R}^n \ 
\mbox{such that}\, \int e^f d\mu = 1 
\end{equation}
with some constant $c>0$ depending on the measure, only. We note that here and all along the paper, the integrals are understood over the whole region.
Under this analytic hypothesis, it was shown in \cite{bobkov-gotze} that
the exponential inequality
\begin{equation} \label{eq:bg}
\int e^f d\mu \leq 
\left( \int e^{\alpha |\nabla f|^2} d\mu \right)^\frac{c}{\alpha - c}, 
\end{equation}
holds true for all smooth $f$ on $\mathbb{R}^n$ with $\int f d\mu = 0$
and for an arbitrary value $\alpha > c$. Here $\nabla f$ is the 
gradient of $f$ and $|\cdot|$ stands for the Euclidean length.
 
A classical example of \eqref{eq:log-sobolev} is provided by
the standard Gaussian measure $\mu = \gamma_n$ with density 
$$
\varphi_n(x) \coloneqq (2\pi)^{-\frac{n}{2}}\,e^{-|x|^2/2},
$$ 
for which $c=1/2$ is optimal \cite{stam59,gross75}. More generally, 
according to the Bakry-Émery criterion \cite{bakry-emery}, the relation 
\eqref{eq:log-sobolev} with constant $c=1/(2\rho)$ holds for any 
$\mu$ with density $e^{-V(x)}$ satisfying 
$\mathrm{Hess}(V) \geq \rho > 0$ (as a matrix). 
We refer the interested reader to 
the textbooks and monographs 
\cite{bakry,ane,ledoux,guionnet-zegarlinski}
for an introduction to the log-Sobolev inequality and its numerous 
connections 
with other fields of mathematics (convex geometry, information theory, 
statistical mechanics...). 

In \cite{bobkov-gotze}, the authors deal with a more general setting of 
a metric space $\mathcal{X}$ in place of $\mathbb{R}^n$ and some "derivation" 
operator in place of $|\nabla f|$. The inequality \eqref{eq:bg} also holds 
in the discrete setting, for example, on a graph for an appropriate Dirichlet 
form in the associated log-Sobolev inequality.

Since $c=1/2$ for the Gaussian measure, the admissible range of the parameter 
$\alpha$ is $(1/2,\infty)$. As for the critical value,
it was observed by Talagrand that $\int e^f d\gamma_n < \infty$ whenever 
$\int e^{\frac{1}{2}|\nabla f|^2}d\gamma_n < \infty$
(which is mentioned without proof in \cite[after Corollary 2.2]{bobkov-gotze}).
A quantitative form of this statement has been recently obtained by 
Ivanisvili and Russell \cite{IR}, by proving that
\begin{equation} \label{eq:IR}
\log \int e^f d\gamma_n \leq  10 
\int \frac{e^{\frac{1}{2} |\nabla f|^2}}{1 + |\nabla f|} d\gamma_n  \quad 
\mbox{for all smooth} \ f \ \mbox{with} \int f d\mu = 0. 
\end{equation}
This result nicely complements the family of 
inequalities \eqref{eq:bg} by dealing with the extremal value $\alpha=1/2$.

A second closely related observation, which was another starting point 
of our investigation, is due to Cianchi, Musil, and Pick. In \cite[Theorem 1.1]{CMP}, 
it was shown that, if $f$ has $\gamma_n$-mean zero and satisfies 
$\int e^{\frac{1}{\kappa^{\beta}}|\nabla f|^\beta} d\gamma_n < M$ for some $\beta >0$ and $M>1$, then 
with some constant $C(M)$ depending on $M$ only, we have
\begin{equation} \label{eq:cmp}
\int \exp\Big\{ |f|^\frac{2\beta}{\beta + 2}\Big\}\, d\gamma_n < C(M)
\end{equation}
as soon as 
\begin{equation} \label{eq:kbeta}
\kappa \leq \kappa_\beta := \frac{1}{\sqrt{2}} + \frac{\sqrt{2}}{\beta} .
\end{equation}
Here, the particular value $\beta=2$ also leads to a refined form of 
Talagrand's observation,
$$
\int e^{|f|}\, d\gamma_n < C(M), \; \mathrm{whenever}\;\;
 \int e^{\frac{1}{2}|\nabla f|^2}d\gamma_n<M,
$$
although \eqref{eq:IR} is more quantitative for such parameter $\beta$.
In \cite{CMP20}, Cianchi, Musil, and Pick show improved bounds in the sense that the integrand of \eqref{eq:cmp} has a faster growth, i.e. they proved that if $\phi:[0,\infty)\rightarrow[0,\infty)$ is an increasing function that diverges to $\infty$ as $t\rightarrow \infty$ with a sufficiently mild growth, then $\int \exp\Big\{( |f|)^\frac{2\beta}{\beta + 2}\Big\}\, \phi(|f|) d\gamma_n < C(M)$. In a subsequent paper \cite{CMP21}, the same authors  derive similar bounds under the constraint 
$\int \exp\{\kappa'|L f|^\beta\}\, d\gamma_n < M$ in place of 
$\int \exp\{\frac{1}{\kappa^\beta}|\nabla f|^\beta\}\, d\gamma_n < M$ with a different 
exponent than $2\beta/(\beta+2),$ and where $L = \Delta - x \cdot \nabla$ is the Ornstein-Uhlenbeck operator.

One of the main features of all these results is that they are dimension 
free, unlike the Moser-Trudinger inequality in the Euclidean space over
the Lebesgue measure. We refer the reader to the introduction of \cite{CMP,CMP21}, 
or to the textbook \cite{saloff-coste} for more information on Moser-Trudinger 
inequality, a historical presentation and references.

It is therefore natural to consider on the Euclidean space equipped 
with a probability measure $\mu$ the following general exponential inequality
\begin{equation} \label{eq:exp}
\int e^f d\mu \leq F \left( \int G \left( |\nabla f| \right) d\mu\right) \quad  
\mbox{for all smooth} \ f \ \mbox{with} \int fd\mu = 0,
\end{equation}
where $F, G \colon [0, \infty) \to [0,\infty)$ are given non-decreasing
functions. More general spaces may also be involved
in this family of Sobolev-type relations.

In this paper our aim is to establish inequalities of the type \eqref{eq:exp} 
in different settings including the framework of the $\Gamma_2$ formalism.
First, under the so-called $\Gamma_2$ condition (which is stronger than 
the log-Sobolev inequality), and using the semi-group technique, we sharpen 
\eqref{eq:bg}, by proving a similar relation with a better exponent.
Moreover, we obtain local inequalities for the semi-group in place 
of the measure $\mu$. Then, we extend the approach of \cite{bobkov-gotze} 
to the class of sub-Gaussian measures under a weaker hypothesis in the form
of a modified log-Sobolev inequality. This will lead us to new exponential 
inequalities for measures with sub-Gaussian tails like the ones with densities 
proportional to $\exp\{-|x|^p\}$, $p \in (1,2)$ (on each fibre).

On the other hand, we introduce a simple and direct argument, essentially based 
on convex duality, to recover the inequality \eqref{eq:IR} for the Gaussian measure, 
in a slightly modified form. One interesting point is that, though 
being totally different, our approach leads to the same conclusion as in
\cite{IR}. We believe that one cannot improve such a relation (except 
for the numerical constant 10). In Section \ref{sec:first} we give more on this 
intuition. 

Our dual argument reveals to be robust enough to be able to give a quantitative 
form of Cianchi, Musil, and Pick  \cite{CMP} for some range of the parameter 
$\beta$, namely $\beta \in (\sqrt{5}-1,2)$, when $\kappa=\kappa_\beta$.
Finally, we will deal in the final section with the discrete setting.


\section{Perturbation, stability, optimality and Reduction} \label{sec:first}

This section is preparatory. Here, we first prove that the 
inequality \eqref{eq:exp} 
is stable under bounded perturbation of the density of the 
probability measure $\mu$. This property is commonly referred to as 
the Holley-Stroock perturbation result. Using the Caffarelli contraction 
theorem, we also show that \eqref{eq:exp}, being valid for the Gaussian 
measure, is extended to the similar relation for any contraction of
$\gamma_n$. Then we will show that the constraint $\int fd\mu = 0$ can be changed,
under mild assumption, into $m_f=0$, where $m_f$ is the median of $f$.

Finally, using Ehrhard's rearrangement technique, we prove that the
inequality \eqref{eq:exp} is by essence one dimensional and may be further 
reduced to the half-line for the class of non-decreasing functions satisfying 
$f(0)=0$ (Hardy-type constraints), in some specific case.

\subsection{Bounded perturbation}

The main result of this section is the following:

\begin{thm} \label{th:HS}
Given $F, G \colon [0, \infty) \to [0,\infty)$, where $F$ is increasing, and let
$\mu$ be a probability measure on $\mathbb{R}^n$ satisfying
$$
\int e^f d\mu \leq F \left( \int G \left( |\nabla f| \right) d\mu \right)
$$
for all smooth $f$ such that $\int f d\mu = 0$. Let $\nu$ be a probability measure 
absolutely continuous with respect to $\mu$ with relative density $h = d\nu/d\mu$
such that $a \leq h(x) \leq b$ for all $x \in \mathbb R^n$ with some 
constants $0<a<b<\infty$. Then, for any smooth $f$ such that $\int fd\nu =0$,
$$
\int e^f d\nu \leq \widetilde F\left(\int G\left(|\nabla f|\right) d\nu \right),
\quad \widetilde F(t) \coloneqq 1+ b F(t/a), \ t \geq 0.
$$
\end{thm}

The argument employs the following known lemma whose proof we give for completeness.

\vskip5mm
\begin{lem} \label{lem:inf}
Let $\Phi \colon \mathbb{R} \to \mathbb{R}$ be a $C^1$-convex function, 
and let $\mu$ a probability measure on $\mathbb{R}$. For any 
$\mu$-integrable function 
$f \colon \mathbb{R} \to \mathbb{R}$, such that $\Phi(f)$ is $\mu-$integrable,
$$
\int \Phi(f)\, d\mu - \Phi \left( \int fd\mu \right)
= \inf_{ t \in \mathbb{R}} \int \left(\Phi(f)- \Phi(t) - \Phi'(t)(f-t) \right) d\mu .
$$
\end{lem}

\begin{proof}
For $t\in\mathbb{R},$ by the Taylor expansion we have that 
\[
\int_{\mathbb{R}^n}(\Phi(f)-\Phi(t)-\Phi'(t)(f-t))d \mu = \int_{\mathbb{R}^n} \Phi(f) d\mu - \Phi(t) - \Phi'(t) \left(\int_{\mathbb{R}^n} f d\mu -t \right)
\]
\[
\geq \int_{\mathbb{R}^n} \Phi(f) d\mu - \Phi \left(\int_{\mathbb{R}^n} f d\mu \right),
\]
with equality if $t=\int_{\mathbb{R}^n}fd\mu.$

\end{proof}

\begin{proof}[Proof of Theorem \ref{th:HS}]
Let $f$ be a smooth function satisfying $\int f d\nu =0$. 
Applying Lemma \ref{lem:inf} (twice) to $\Phi(x)=e^x$, we have
\begin{align*}
\int e^f d\nu - 1
& =
\int \Phi(f)\, d\nu - \Phi \left( \int fd\nu \right) \\
& = 
\inf_{t \in \mathbb{R}} \int \left(\Phi(f)- \Phi(t) - \Phi'(t)(f-t) \right) d\nu \\
& \leq 
b\, \inf_{t \in \mathbb{R}} \int \left(\Phi(f)-\Phi(t) - \Phi'(t)(f-t)\right) d\mu
 \, = \,
b\, \left( \int e^f d\mu - \exp \left\{ \int f d\mu \right\} \right) .
\end{align*}
Here, we used the convexity of $\Phi$ which ensures that 
$\Phi(f)-\Phi(t) - \Phi'(t)(f-t) \geq 0$. 
Since $\exp \left\{ \int f d\mu \right\} \geq 0$, it follows that
$$
\int e^f d\nu \leq 
1 + b\, F \left( \int G \left( |\nabla f| \right) d\mu \right) \leq 
1 + b\, F \left( \frac{1}{a} \int G \left( |\nabla f| \right) d\nu \right).
$$
\end{proof}

\subsection{Stability by contraction}

A probability measure $\mu$ on $\mathbb R^n$ is often said to be strongly log-concave,
if it has a log-concave density with respect to the standard Gaussian measure,
that is, when $\mu(dx) = e^{-V(x)}\,d\gamma_n(x)$ with $V$ convex.

\begin{thm}[Caffarelli \cite{caffarelli2000monotonicity, caffarelli2002monotonicity}] \label{thm:Caffarelli} 
If $\mu$ is strongly log-concave on $\mathbb{R}^n$, 
then there exists a $1$-Lipschitz map $T \colon \mathbb{R}^n \to \mathbb{R}^n$ satisfying $\int h(T)\, d\gamma_n = \int h\, d\mu$ for all bounded measurable
functions $h$ on $\mathbb{R}^n$.
\end{thm}

In other words, $T$ pushes forward $\gamma_n$ to $\mu$, or $\mu$ appears as
image of $\gamma_n$ under $T$. With the help of this theorem, we obtain 
the following elementary, but useful result.

\begin{prop}
Let $F,G \colon [0, \infty) \to [0,\infty)$ be non-decreasing and continuous, and let 
$\mu$ be a strongly log-concave probability measure on $\mathbb{R}^n$. 
Assume that for all smooth $f$ with $\int fd\gamma_n=0,$ it holds that
$$
\int e^f d\gamma_n \leq F \left( \int G \left( |\nabla f| \right) d\gamma_n \right),
$$
then
$$
\int e^g d\mu \leq F \left( \int G \left( |\nabla g| \right) d\mu \right)
$$
for all smooth $g$ with $\int g d\mu =0$.
\end{prop}

\begin{proof}
The hypothesis about $\gamma_n$ is extended to all locally Lipschitz
functions $f$ with $\gamma_n$-mean zero, in which case the modulus of 
the gradient may be defined by
$$
|\nabla f(x)| = \limsup_{y \rightarrow x} \frac{|f(x) - f(y)|}{|x-y|}, \quad
x \in \mathbb R^n.
$$
Let $T$ be the Lipschitz map from Theorem \ref{thm:Caffarelli}, and let $g$ be
a smooth function with $\mu$-mean zero. Note that by a standard approximation argument, we can assume that $g$ is bounded. Then $f=g(T)$ is locally Lipschitz
and satisfies, for all $x \in \mathbb R^n$,
\begin{eqnarray} \nonumber
|\nabla f(x)| 
 & = & 
\limsup_{y \rightarrow x}\, \frac{|g(Tx) - g(Ty)|}{|x-y|} \\ \nonumber
 & = &
\limsup_{y \rightarrow x}\,\bigg[\frac{|g(Tx) - g(Ty)|}{|Tx-Ty|} \,
\frac{|Tx - Ty|}{|x-y|}\bigg] \, \leq \, |\nabla g|(Tx).
\end{eqnarray}
Applying the exponential inequality to $f$ with respect to $\gamma_n$, we get
\begin{eqnarray} \nonumber
\int e^g d\mu 
 & = &
\int e^{f} d\gamma_n \, \leq \,
F\left(\int G\left(|\nabla f|\right) d\gamma_n \right) \\ \nonumber
 & \leq &
F \left( \int G \left( |\nabla g| (T)\right) d\gamma_n  \right) \, = \,
F \left( \int G \left( |\nabla g| \right) d\mu  \right).
\end{eqnarray}
\end{proof}

\subsection{Zero mean versus zero median}
While the inequality \eqref{eq:exp} appears under the condition 
$\int f\,d\mu=0$, it is useful to know if one may replace with
a similar condition such as $m_f=0$, where 
$$
m_{f}:=\mathrm{inf}\left\{t\in\mathbb{R}: \mu(\{x\in\mathbb{R}^{n}: f(x)>t\})\leq \frac{1}{2}\right\}
$$ 
denotes the maximal median of $f$
under the measure $\mu.$  Here we describe one specific example,
in which the answer is affirmative. Namely, consider 
$$
F(t)=e^t \quad \mbox{and} \quad G(t) = \frac{1}{1+t}\,e^{t^2/2}, 
$$
therefore dealing with the inequality \eqref{eq:IR} (but the computations 
to come can easily be adapted to many other examples, including $F$ and $G$ 
as in \eqref{eq:bg} for instance).

Recall that a probability measure $\mu$ on $\mathbb{R}^n$ is said to satisfy the Maz'ya-Cheeger inequality if there exists a constant $c \in (0,\infty)$ such that 
for all $f$ smooth,
\begin{equation} \label{eq:cheeger}
\int |f - m_f|\, d\mu \leq c \int |\nabla f| d\mu.
\end{equation}
As is well-known (see \cite{maz'ja60, maz2011}), this hypothesis may be equivalently stated in terms of the 
isoperimetric-type inequality
$$
\min\{\mu(A),1-\mu(A)\} \leq c \mu^+(A),
$$
relating the $\mu$-perimeter 
$\mu^+(A) = \liminf_{\ep \downarrow 0}\,(\mu(A_\ep) - \mu(A))/\ep$
to the $\mu$-size of an arbitrary Borel set $A$ in $\mathbb R^n$
(where $A_\ep$ denotes an $\ep$-neighborhood of $A$). 
It is stronger than a Poincar\'e-type, although in general it is not comparable to  the log-Sobolev inequality for the measure $\mu$

Now assume that, for some $a \in (0,\infty)$,
$$
\log \int e^f d\mu \leq a \int G(|\nabla f|)\, d\mu  
\quad \mbox{for all } f \mbox{ with } \int fd\,\mu=0.
$$
Consider a smooth function $g$ with $m_g = 0$.
Applying the latter inequality to $f = g - \int gd\mu$, we get
$$
\log \left( \int e^g d\mu \right) \leq  
\int g\, d\mu + a  \int G \left( |\nabla g|\right) d\mu .
$$
By Maz'ya-Cheeger's inequality \eqref{eq:cheeger},
$$
\int g d\mu \leq \int |g-m_g\,|\, d\mu \leq c \int |\nabla g|\, d\mu.
$$
Therefore, since $t \leq 2e^{t^2/2}/(1+t)$ for all $t \geq 0$, we end up with
$$
\log \left( \int e^g d\mu \right) 
\leq 
(a+2c)  \int G \left( |\nabla g|\right) d\mu.
$$

The same argument in fact shows that the two conditions are equivalent in \eqref{eq:IR}, at the expense of some loss in the constant.

\subsection{Optimality in \eqref{eq:IR}}

Recall that the inequality \eqref{eq:IR} states that, for any $f$ 
with $\int f d\gamma_n = 0$,
$$
\log \left( \int e^f d\gamma_n \right) \leq 
10 \int \frac{e^{|\nabla f|^2/2}}{1+|\nabla f|} d\gamma_n
$$
We will develop an alternative approach, based on a dual convexity argument, 
leading to
$$
\log \left( \int e^f d\gamma_n \right) 
\leq 
8\int \frac{e^{|\nabla f|^2/2}}{\sqrt{1+(|\nabla f|^2/2)}} d\gamma_n .
$$
Let us show that such choice of the function 
$G(t)=e^{t^2/2}/\sqrt{1+t^2/2}$ appearing in the integrand on the right-hand side
is essentially optimal. Restricting ourselves to the one dimensional case,
suppose that, for all bounded, locally Lipshitz functions $f$ on the real line,
\begin{equation} \label{eq:expbis}
\log \left( \int e^f d\gamma \right) 
\leq 
\int f d\gamma + 
F \left( \int \frac{e^{{f'}^2/2}}{H\left( {f'} \right)} d\gamma \right) 
\end{equation}
for some increasing function $H \geq 1$, where $\gamma = \gamma_1$, and
$F$ is non-decreasing. For the particular functions
$f_N(x)= \frac{1}{2}\,(x \wedge N)^2$, the left-hand side satisfies
$$
\log \left( \int e^{f_N} d\gamma \right) \geq 
\log \left( \int_{-N}^N \frac{dx}{\sqrt{2\pi}} \right)  \geq \log(2N) -1,
$$
while the right-hand side can be bounded above by
$$
\frac{1}{2} + F \left( 2\int_0^N \frac{1}{H(t)}\, dt + 1 \right),
$$
since $\int f_N d\gamma \leq \frac{1}{2} \int x^2 d\gamma = \frac{1}{2}. $ Therefore, the inequality \eqref{eq:expbis} cannot hold if $1/H$ is integrable 
at infinity. This excludes, for example, $H$ behaving like 
$t (\log t)^{1+\varepsilon}$ for large $t$. Therefore, heuristically 
the biggest admissible function $H$ lies "between" $t$ and $t \log t$ at infinity.
Observe that the above example does not exclude an inequality of the type
$$
\log \log \left( \int e^f d\gamma \right) 
\leq 
F \left( \int \frac{e^{{f'}^2/2}}{(1+|f'|) \log (1+|f'|)} d\gamma \right)
$$
that we don't know if it could hold, or not.

On the other hand, for $H(t) = 1+|t|$, the above example shows that 
$F(t) \geq t$ in the large, and therefore the logarithm is needed in 
the left hand side of \eqref{eq:IR}, in contrast with the inequality \eqref{eq:bg}.

To conclude this section, we observe that for the special choice $H(x)=1$, linear 
perturbation of the quadratic example used above shows that
an inequality of the form
$$
\int e^f d\gamma \leq \left( \int e^{{f'}^2/2} d\gamma \right)^2 
$$
could hold for all $f$ with mean zero with respect to the Gaussian measure.

\subsection{Reduction to dimension 1}

In this section we show that, under mild assumption on $F$ and $G$, the inequality 
\eqref{eq:exp} for the standard Gaussian measure $\gamma_n$ in dimension $n$ 
holds if and only if it holds in dimension 1 for $\gamma_1$. 
Therefore, by essence this inequality for the standard Gaussian is one dimensional. 
To achieve this, one could use the localization technique
of Lov\'asz and Simonovits \cite{Lovasz} (see also \cite{KLS,guedon}). 
Here we will instead use the Gaussian Rearrangement technique of 
Ehrhard \cite{ehrhard, ASENS_1984_4_17_2_317_0}.

\begin{prop} \label{prop:product}
Let $F, G \colon [0, \infty) \to [0,\infty)$ be non-decreasing. Assume that 
$G$ is convex and that $f$ is smooth. Then, the exponential inequality 
$$
\int e^f d\gamma_n \leq  F \left( \int G \left( |\nabla f| \right) d\gamma_n \right), 
\quad \int f d\gamma_n = 0,
$$
holds in any dimension if and only if it holds in dimension $n=1$
for the class of all non-decreasing smooth functions $f$.
\end{prop}

\begin{remark}
The same result holds if one replaces the left integral $\int e^f d\gamma$ 
by $\int H(f)\,d\gamma_n$, where $H$ is non-negative.   
\end{remark}

\begin{proof}
The Gaussian rearrangement of a Borel set $A$ of $\mathbb{R}^n$ is 
the half space 
$$
A^*= \{x_1 > \lambda \} \subset \mathbb{R}^n, 
$$
where the number $\lambda$ is chosen so that $A$ and $A^*$ have the same Gaussian 
measure $\gamma_n(A)=\gamma_n(A^*)=\gamma_1([x_1,\infty))$. Put 
$\Phi(x)=\int_{-\infty}^x \varphi(t)\,dt$, where $\varphi$ denotes 
the standard normal density. More generally, we define the Gaussian rearrangement 
of a measurable function $f \colon \mathbb{R}^n \to \mathbb{R}$
as a non-decreasing function
$$
f^*(x) \coloneqq \inf \{ a \in \mathbb{R}: 
\Phi(x_1) \leq \gamma_n(\{f \leq a \}) \} , \quad x=(x_1,x_2,\dots,x_n) \in \mathbb{R}^n.
$$
It is clear that $f^*$ depends only on the first coordinate. With a slight 
abuse of notation we will write $f^*$ also for the function defined on the real 
line, i.e. $f^*(x_1)=f^*(x_1,x_2,\dots,x_n)$ for an arbitrary 
choice of $x_2, \dots,x_n$. The main feature of the Gaussian rearrangement is 
that $f$ and $f^*$ are equimeasurable, therefore 
$$
\int_{\mathbb{R}^n} e^f d\gamma_n = \int_\mathbb{R}  e^t \gamma_n(\{f>t\})dt= \int_\mathbb{R} e^t \gamma_n(\{f^*>t\})dt= \int_{\mathbb{R}^n} e^{f^*} d\gamma_n 
= \int_{\mathbb{R}} e^{f^*} d\gamma_1.
$$
Moreover, for any non-decreasing convex function 
$G \colon [0,\infty) \to \mathbb{R}$, the following P\'olya-Szego type inequality
$$
\int_{\mathbb{R}^n} G(|\nabla f|)\, d\gamma_n \geq 
\int_{\mathbb{R}^n} G(|\nabla f^*|)\, d\gamma_n = 
\int_{\mathbb{R}} G(|{f^*}'|)\, d\gamma_1 
$$
holds true for any smooth $f$. Since
$\int f d\gamma_n = \int f^* d\gamma_1 = 0$, assuming that the exponential 
integral inequality holds in dimension $n=1$, we get
$$
\int_{\mathbb{R}^n} e^f d\gamma_n \, = \,
\int_{\mathbb{R}} e^{f^*} d\gamma_1 \, \leq \,
F \left( \int_{\mathbb{R}} G(|{f^*}'|)\, d\gamma_1 \right) \, \leq \,
F \left( \int_{\mathbb{R}^n} G(|\nabla f|)\, d\gamma_n \right),
$$
by the monotonicity of $F$ in the last step.
\end{proof}


\subsection{Reduction to the half line, monotone functions, and 
a Hardy-type inequality}

The aim of this section is to show that one can relate the one-dimensional 
version of the inequality \eqref{eq:exp} to an inequality of Hardy type 
on the half-line. Furthermore, at the expense of some loss in the constants, 
it is proved that in the one dimensional version of \eqref{eq:exp} 
one can restrict oneself to monotone functions.
For simplicity, we will only deal with symmetric probability measures.
We will use the following notation: given a function $F$ on $[0,\infty)$,
\begin{equation} \label{eq:F2}
F_2(x) := \sup
\big\{F(x_1) + F(x_2): x_1,x_2 \geq 0, \ x_1 + x_2 = x\big\}, \quad x \geq 0.
\end{equation}
If $F$ is non-decreasing, then $F_2(x) \leq 2 F(x)$, and if in addition $F$ 
is convex, then necessarily $F_2(x) = F(x)+F(0)$ for all $x \geq 0$.

We say that $\mu$ (on the line) satisfies Hardy's inequality if
\begin{equation} \label{eq:hardy}
\int_0^\infty |f-f(0)|\, d\mu \leq A \int_0^\infty |f'|\, d\mu  
\end{equation}
for $C^1$ functions $f$ on the half-line with some constant $A \in (0,\infty)$.

\begin{lem} \label{lem:reduction}
Let $F, G \colon [0, \infty) \to [0,\infty)$ be non-decreasing and let 
$\mu$ be a symmetric probability measure on the line. Assume that for 
all non-decreasing smooth enough functions $f \colon [0,\infty) \to \mathbb{R}$, 
\begin{equation} \label{eq:reduction}
\int_0^\infty e^f d\mu - \frac{1}{2}e^{f(0)}  
\leq  
e^{f(0)}  F \left( \int_0^\infty G(|f'|)\, d\mu \right) .
\end{equation}
Then, for all smooth $g \colon \mathbb{R} \to \mathbb{R}$ with $\int G(|g'|)d\mu<\infty,$
$$
\int_{-\infty}^\infty e^g\, d\mu - 
\exp\Big\{\int_{-\infty}^\infty g\, d\mu\Big\}
\leq e^{g(0)}\, F_2 \left(\int_{-\infty}^\infty G(|g'|)\, d\mu \right) .
$$
Moreover, if $\mu$ satisfies the Hardy-type inequality \eqref{eq:hardy}, then, 
for all $g$ with $\int_{-\infty}^\infty g\, d\mu = 0$, 
\begin{equation} \label{eq:final}
\int_{-\infty}^\infty e^g d\mu \leq  1 +
\exp\Big\{A\int_{-\infty}^\infty |g'|\,d\mu\Big\}\, F_2 \left( \int_{-\infty}^\infty G(|g'|) d\mu \right) .
\end{equation}
\end{lem}

\begin{remark} 
{\rm By symmetry, $\mu(-\infty,0)=\mu(0,\infty)=1/2$. This explains 
the factor $1/2$ appearing in \eqref{eq:reduction}. Observe also that 
this inequality is invariant by changing $f$ into $f+c$ for any constant $c$.
The inequality \eqref{eq:reduction} can be seen as an exponential Hardy-type 
inequality.}
\end{remark}

\begin{proof}
Fix $g \colon \mathbb{R} \to \mathbb{R}$ and assume without loss of 
generality that $g(0)=0$. Applying Lemma \ref{lem:inf} with $\Phi(x)=e^x$ 
and $t=g(0)=0$, we have
\begin{align*}
\int_{-\infty}^\infty e^g d\mu 
- \exp\Big\{\int_{-\infty}^\infty g\, d\mu\Big\}
\leq 
\left( \int_0^\infty e^g d\mu - \frac{1}{2} - 
\int_0^\infty g\, d\mu  \right) +
\left( \int_{-\infty}^0 e^g d\mu - \frac{1}{2} - 
\int_{-\infty}^0 g\, d\mu \right)
\end{align*}
Define
$$
f_+(x) := \int_0^x g'(t)\, \mathds{1}_{g'(t)>0}\, dt ,
\quad 
f_-(x) := -\int_{-x}^0 g'(t)\, \mathds{1}_{g'(t)<0}\, dt, \quad x \geq 0 .
$$
By construction, both $f_+$ and $f_-$ are non-negative, non-decreasing, 
and satisfy $g \leq f_+$ on the positive axis, while $g(-x) \leq f_-(x)$ for all $x \geq 0$. Furthermore, $f'_+(x) \leq |g'(x)|$ and $f_-'(x) \leq |g'(-x)|$, $x \geq 0$, and for any measurable function $H$ on the line, by symmetry of $\mu$, $\int_{-\infty}^0 Hd\mu=\int_0^\infty H(-x) d\mu$.
Therefore, since $x \mapsto e^x-x$ is increasing on $(0,\infty)$, we have
\begin{align*}
\int_{-\infty}^\infty e^g d\mu 
- e^{\int_{-\infty}^\infty g d\mu}
& \leq 
\left( \int_0^\infty e^{f_+} d\mu - \frac{1}{2} - \int_0^\infty f_+ d\mu \right) 
+
\left( \int_0^{\infty} e^{f_-} d\mu - \frac{1}{2} - \int_0^\infty f_- d\mu  \right) \\
& \leq 
\left( \int_0^\infty e^{f_+} d\mu - \frac{1}{2}  \right) 
+
\left( \int_0^{\infty} e^{f_-} d\mu - \frac{1}{2}   \right)
\end{align*}
where for the last inequality we used the fact that $f_\pm \geq 0$.  Thanks to our assumption, we finally get
\begin{align*}
\int_{-\infty}^\infty e^g d\mu 
- e^{\int_{-\infty}^\infty g d\mu}
& \leq 
F \left( \int_0^\infty G(f'_+) d\mu \right) 
+
F \left( \int_0^\infty G(f_-') d\mu \right) \\
& \leq 
F_2 \left( \int_0^\infty G(f_+') d\mu  + \int_0^\infty G(f_-') d\mu \right) \\
& 
\leq 
F_2 \left( \int_0^\infty G(|g'|) d\mu  + \int_{0}^{\infty} G(|g'(-x)|) d\mu \right)
\, = \,
F_2 \left( \int G(|g'|) d\mu \right)
\end{align*}
where we used that $F_2$ is non-decreasing (which is a consequence of 
the fact that $F$ is non-decreasing). This proves the first part of the lemma.

For $g \colon \mathbb{R} \to \mathbb{R}$  with $\int_{-\infty}^\infty gd\mu=0$ we have just 
proved that
$$
\int_{-\infty}^\infty e^g d\mu \leq 
1 + e^{g(0)} F_2 \left( \int_{-\infty}^\infty G(|g'|) d\mu \right) =
1 + e^{- \int_{-\infty}^\infty (g-g(0)) d\mu} F_2 \left( \int_{-\infty}^\infty G(|g'|) d\mu \right).
$$
Hardy-type inequality \eqref{eq:hardy} and the symmetry of $\mu$ guarantee that
$$
\int_{-\infty}^\infty |g-g(0)| d\mu =\int_0^\infty |g(x)-g(0)| d\mu +\int_0^\infty |g(-x)-g(0)| d\mu \leq A \int_{-\infty}^\infty |g'| d\mu .
$$
Therefore
\begin{align*}
\int_{-\infty}^\infty e^g d\mu 
& \leq 
1 + e^{\int_{-\infty}^\infty |g -g(0)|d\mu} F_2 \left( \int_{-\infty}^\infty G(|g'|) d\mu \right) \\
& \leq 
 1 + e^{A \int_{-\infty}^\infty |g'|d\mu} F_2 \left( \int_{-\infty}^\infty G(|g'|) d\mu \right) 
\end{align*}
leading to the second part of the lemma.
\end{proof}

We end this section by focusing on the Gaussian measure $\gamma_1$ and 
the special case $F(x)=ae^{bx}-c$ for some constants $a,b,c>0$. The lemma 
will be applied later on to $G(x)=e^{x^2/2}/\sqrt{1+(x^2/2)}$, which amounts 
to considering the inequality \eqref{eq:IR} since $G$ compares with 
$e^{x^2/2}/(1+x)$. The choice of $G$ is governed by the fact that it is 
convex increasing, a property that is not shared by the map  
$x \mapsto e^{x^2/2}/(1+x)$ considered in \cite{IR}.

\begin{lem} \label{lem:paris}
Let $F(x)=ae^{bx}-c$ for $a,b,c>0$ with $a \geq c$, and let $G \colon [0, \infty) \to [0,\infty)$ be non-decreasing convex function.
Assume that the Gaussian measure $\gamma_1$ satisfies
$$
\int_0^\infty e^f d\gamma_1 - \frac{1}{2}e^{f(0)} \leq 
e^{f(0)} F \left( \int_0^\infty G(|f'|) d\gamma_1 \right)
$$
for all smooth non-decreasing $f \colon [0, \infty) \to [0,\infty)$.  
Then, for any dimension $n$ and any smooth $g \colon \mathbb{R}^n \to \mathbb{R}$ 
such that $\int_{-\infty}^\infty gd\gamma_n = 0$, the following inequality holds
\begin{equation} \label{eq:13}
\int_{-\infty}^\infty e^g d\gamma_n 
\leq 
1 + a \exp \left\{(d+b)\int_{-\infty}^\infty G(|\nabla g|) d\gamma_n \right\}
+(a-2c) \exp \left\{d\int_{-\infty}^\infty G(|\nabla g|) d\gamma_n \right\}
\end{equation}
with
$d \coloneqq \sqrt{\frac{\pi}{2}} \max_{x \geq 0} \frac{x}{G(x)}$. 
\end{lem}

\begin{remark}
Observe that if $d=\infty,$ the right-hand side of \eqref{eq:13} is infinite, and the conclusion of the lemma is useless.
\end{remark}
\begin{proof}
First we claim that the Hardy-type inequality \eqref{eq:hardy} holds for 
$\gamma_1$ with constant $A=\sqrt{\frac{\pi}{2}}$. 
To prove that this constant is best possible, we may use a result by 
Muckenhoupt \cite{muckenhoupt}. Indeed, this author proved that 
the best constant in \eqref{eq:hardy} is 
$$
A = \sup_{r > 0}\, e^{r^2/2} \int_r^\infty e^{-x^2/2} dx .
$$
Now observe that
$$
\int_r^\infty e^{-x^2/2} dx 
\leq 
\frac{1}{r}\int_r^\infty x e^{-x^2/2} dx
=
\frac{e^{-r^2/2}}{r} .
$$
Therefore, if we denote $H(r) = e^{r^2/2} \int_r^\infty e^{-x^2/2} dx$, $r>0$, 
we have
$$
H'(r) = re^{r^2/2} \int_r^\infty e^{-x^2/2} dx -1 \leq 0 .
$$
In turn, the best constant is $A = H(0) = \sqrt{\frac{\pi}{2}}$ as announced.

Now Lemma \ref{lem:reduction} implies that 
$$
\int_\mathbb{R} e^g d\gamma_1 
\leq 
1 + e^{A \int |g'|d\gamma_1} F_2 \left( \int_\mathbb{R} G(|g'|)d\gamma_1 \right)
$$
holds for all $g  \colon \mathbb{R} \to \mathbb{R}$ smooth with $\int gd\gamma_1=0$.
Observe that, for $F(x)=ae^{bx}-c$, $F_2(x)=F(x)+F(0)=ae^{bx}+a-2c$. Hence, 
the latter inequality reads
\begin{align*}
\int_{-\infty}^\infty e^g d\gamma_1 
& \leq 
1 + e^{A \int_{-\infty}^\infty |g'|d\gamma_1} 
\left( a e^{ b \int_{-\infty}^\infty G(|g'|)d\gamma_1 } + a - 2c \right) \\
& \leq 
1 + a \exp \left\{(d+b)\int_{-\infty}^\infty G(| g'|) d\gamma_1 \right\}
+(a-2c) \exp \left\{d\int_{-\infty}^\infty G(|g'|) d\gamma_1 \right\} .
\end{align*}
The expected result follows from Proposition \ref{prop:product}.
\end{proof}


\section{Exponential inequality via $\Gamma_2$ condition and semi-group}
\label{sec:semigroup}

In this section we obtain some exponential inequality using the so-called 
$\Gamma_2$-condition and via semi-group techniques. Let us start by collecting 
some useful and well known facts on the $\Gamma_2$ calculus. We refer 
to \cite{bakry}, \cite[Chapter 5]{ane} and to 
the excellent \cite[Section 2]{bakry-ledoux} for more details and comments. 

\subsection{$\Gamma_2$ formalism} \label{sec:gamma2}
The general setting is  given by an abstract Markov generator $L$, on some 
probability space $(\mathcal{X},\mathcal{B},\mu)$, associated to the semi-group 
$(P_t)_{t \geq 0}$. We assume that $L$ is self-adjoint in $\mathbb{L}^2(\mu)$ 
and has domain $\mathcal{D}(L)$. We further assume the existence of a set 
$\mathcal{A}$ of a good family of functions (see \cite[Definition 2.4.2]{ane}). 
The set $\mathcal{A}$ is assumed to be contained and to be dense in all 
$\mathbb{L}^p(\mu)$, $1 < p <\infty$, it is stable under the action of $P_t$, 
$t>0$, and $L$ and also by composition by any $\mathcal{C}^\infty$-smooth functions. 
Finally we assume that $\mathcal{A}$ contains all constant functions and 
$\lim_{t \to \infty} P_tf = \int f d\mu$ for all $f \in \mathcal{A}$.

For the Ornstein-Uhlenbeck process $\mathcal{A}$ may consist of all 
$\mathcal{C}^\infty$ smooth functions with successive derivatives slowly 
growing at infinity
(\textit{i.e.}\ such that $|f| \leq P$ at infinity, for some polynomial $P$ 
and the same for the derivatives). For a compact connected complete Riemannian 
manifold and $L=\Delta +X$, with $\Delta $ the Laplace-Beltrami operator, and 
$X$ a smooth vector field (with no constant term), see below, $\mathcal{A}$ 
may consist of all $\mathcal{C}^\infty$ smooth functions.

For non-compact Riemanian manifolds the situation is more complicated since 
compactly supported $\mathcal{C}^\infty$-smooth functions are not stable 
under $P_t$ in general. We refer to \cite{elworthy} for more on this issue. 

Now, following P.-A. Meyer and Bakry-Émery, we introduce
the "carr\'e du champs" operator
$$
\Gamma(f,g) \coloneqq \frac{1}{2} \left( L(fg) - fLg - g Lf \right) 
$$
and its iterated
$$
\Gamma_2(f,g) \coloneqq 
\frac{1}{2} \left( L\Gamma(f,g) - \Gamma(f,Lg) - \Gamma(g,Lf) \right), 
\qquad f,g \in \mathcal{A}.
$$
For simplicity we write $\Gamma(f):=\Gamma(f,f)$
and $\Gamma_2(f)=\Gamma_2(f,f)$.

On a connected complete Riemannian manifold $\mathcal{X}=M$, denote by $dx$ and
the Riemannian volume element. For $L = \Delta + \nabla V$,
with $V$ smooth enough and $\int e^{V}dx=1$, 
$\Gamma(f)=|\nabla f|^2$ is the Riemannian length of the gradient. The Bochner's 
Formula indicates that $\Gamma_2(f)=Ric(\nabla f, \nabla f) + \|Hess(f)\|_2^2$
where $Ric$ is the Ricci tensor of $M$ and $\|Hess(f)\|_2^2$ the Hilbert-Schmidt 
norm of the tensor of the second order derivatives of $f$.
The potential $V(x)=-|x|^2/2$ in $\mathbb{R}^n$ corresponds 
to the Ornstein-Uhlenbeck semi-group.

\bigskip 

Next we say that $L$ has curvature $\rho \in \mathbb{R}$ ($\Gamma_2$-condition)
if 
\begin{equation} \label{eq:curvature}
\Gamma_2(f) \geq \rho \Gamma(f), \qquad \forall f \in \mathcal{A} .
\end{equation}

As is well known the Ornstein-Uhlenbeck operator in $\mathbb{R}^n$ has curvature $\rho=1$, the operator
$L = \Delta + \nabla V$, on a manifold with Ricci curvature bounded below by $R$, has curvature bounded below by $\rho$ if  $R + \nabla  \nabla V \geq \rho g$ where $g$ is the Riemannian metric. 

One of the main feature of the above general framework is that it allows to prove similar results in different settings ($\mathbb{R}^n$, manifold etc.).
One important property that illustrate this fact is the commutation of the semi-group and the carr\'e du champs operator. Indeed, it is well-known that the curvature condition \eqref{eq:curvature} is equivalent to saying that, for all $f \in \mathcal{A}$, it holds
\begin{equation} \label{eq:commutaiton}
    \Gamma(P_t f) \leq e^{-2 \rho t} P_t \Gamma(f) .
\end{equation}

Also, the operator $L$ satisfies in our general framework, the following integration by part formula, for $f,g \in \mathcal{A}$,
\begin{equation}\label{eq:ibp}
    \int f L g d\mu = - \int \Gamma(f,g) d\mu .
\end{equation}
In most places, we may need that $\Gamma$ represents a derivation. This is the case when $L$ is a diffusion, which means that, for any $\Phi \colon \mathbb{R}^n \to \mathbb{R}$, $\mathcal{C}^\infty$ and any $f=(f_1,\dots,f_n) \in \mathcal{A}^n$, it holds that
$$
L(\Phi(f)) = \sum_{i=1}^n \partial_i \Phi(f) Lf_i 
+ \sum_{i,j=1}^n \partial^2_{ij}\Phi(f) \Gamma(f_i,f_j) .
$$
Here $\partial_i, \partial^2_{ij}$ are shorthand notation for the first and second order derivatives with respect to the $i$-th, $i,j$-th variables respectively. One crucial consequence of the diffusion property is that
\begin{equation} \label{eq:chainrule} 
\Gamma(\Phi(f),g)=\sum_{i=1}^n \partial_i \Phi(f)\Gamma(f_i,g), \qquad f,g \in \mathcal{A} .
\end{equation} 
The above diffusion property is satisfied by all examples given above.

Finally, we observe that, for diffusions,  \eqref{eq:commutaiton} is also equivalent to the stronger commutation
\begin{equation} \label{eq:commutationbis}
    \sqrt{\Gamma(P_tf)} \leq e^{-\rho t} P_t ( \sqrt{\Gamma(f)}), \qquad f \in \mathcal{A} .
\end{equation}

\subsection{Inequality \eqref{eq:bg} revisited}

In this section we provide a semi-group approach, under the $\Gamma_2$ condition, of Inequality \eqref{eq:bg} that improves the exponent $\frac{c}{\alpha -c}$. 
The result will however still exclude the extremal value $\alpha=c$. On the other hand, since it is classical that the curvature condition $\Gamma_2 \geq \rho \Gamma$, $\rho >0$, implies the log-Sobolev inequality, our result deal with less general situations than in \cite{bobkov-gotze} (there exist probability measures satisfying the log-Sobolev inequality and not the $\Gamma_2$ condition).  Also, our  result represents a local version (in the sense that it involves the semi-group $P_t$)  of \eqref{eq:bg}, which, to the best of our knowledge, is new.
In other words, the approach in \cite{bobkov-gotze} is more robust while our approach by semi-group gives a stronger statement.

\begin{thm} \label{th:BG}
Let $L$ be   some Markov diffusion operator satisfying the curvature condition 
$\Gamma_2(f) \geq \rho \Gamma(f)$, $\rho >0$, for all $f \in \mathcal{A}$.
Then, for all $f$ and all $t>0$ the following inequality holds
$$
\log P_t (e^f) - P_t f  \leq c_\alpha(t) \left( \log P_{t} \left( e^{\alpha \Gamma (f)} \right)  \right)
$$
for all 
$\alpha > \frac{e^{2\rho t} + e^{-2\rho t} -2}{2\rho(e^{2\rho t}-1)}$
with 
$$
c_\alpha(t)= 
\frac{\sqrt{1-e^{-2\rho t}}}{2\sqrt{2\rho \alpha}}
\log \left( \frac{(e^{\rho t}a_t-1)(e^{\rho t}a_t+1+a_t^2-e^{-2\rho t})}
{(e^{\rho t}a_t+1)(-e^{\rho t}a_t+1+a_t^2-e^{-2\rho t})} \right) 
$$
and
$$
a_t^2 \coloneqq 2\rho \alpha (e^{2\rho t}-1) .
$$
In particular, for any $\alpha > \frac{1}{2\rho}$ and any $f$ satisfying $\int fd\mu=0$, it holds
$$
\int e^f d\mu 
\leq \left( \int  e^{\alpha \Gamma(f)} d\mu   \right)^{\frac{1}{2\sqrt{2\rho \alpha}}\log \frac{\sqrt{2\rho \alpha}+1}{\sqrt{2\rho \alpha} -1}} .
$$
\end{thm}

\begin{remark}
Since $\frac{1}{2x}\log\left( \frac{x+1}{x-1} \right) \leq \log \left( \frac{x^2}{x^2-1}\right)$, we immediately get that $$
\frac{1}{2\sqrt{2\rho \alpha}}\log \frac{\sqrt{2\rho \alpha}+1}{\sqrt{2\rho \alpha} -1} \leq \log \frac{2 \rho \alpha}{2 \rho \alpha -1}
 \leq \frac{1}{2\rho \alpha -1}. 
 $$
 The result above is therefore stronger
than \eqref{eq:bg} as announced
(recall that for the $\Gamma_2$-condition implies the log-Sobolev inequality \eqref{eq:log-sobolev} with constant $c=1/(2\rho)$, hence in our setting the exponent in the right hand side of \eqref{eq:bg} is precisely $ \frac{1}{2\rho \alpha -1}$).
\end{remark}

\begin{proof}
The second part of the theorem is a direct consequence of the first, in the limit $t \to \infty$ since 
$\lim_{t \to \infty} P_t e^f =\int e^f d\mu$ and $\lim_{t \to \infty} P_tf = \int fd\mu=0$.

For the first part, observe that 
$$
\log P_t (e^f) - P_t f 
 = - \int_0^t \frac{d}{ds} \log P_{t-s} \left(  e^{P_s f} \right) ds .
$$
By the diffusion property, we have
$$
- \frac{d}{ds} \log P_{t-s} \left(  e^{P_s f} \right)
=
\frac{P_{t-s} \left( Le^{P_sf} - e^{P_sf}LP_sf\right)}{P_{t-s} \left(  e^{P_s f} \right)} 
=
  \frac{P_{t-s} \left( e^{P_sf} \Gamma( P_sf ) \right)}{P_{t-s} \left(  e^{P_s f} \right)} .
$$
Therefore,
$$
\log P_t (e^f) - P_t f 
 =  \int_0^t  \frac{P_{t-s} \left( e^{P_sf} \Gamma( P_sf ) \right)}{P_{t-s} \left(  e^{P_s f} \right)} ds 
 =
\int_0^t P_{t-s} \left(h \Gamma( f_s ) \right)  ds
$$
where we set  for simplicity $h = e^{P_sf} /P_{t-s} \left(  e^{P_s f} \right)$, which is a density with respect to $P_{t-s}$, and $f_s=P_sf$. Applying the entropic inequality (see \textit{e.g.}\ \cite[Chapter 1]{ane}), for any $\theta >0$, it holds that
$$
P_{t-s} ( h \Gamma( P_sf) ) \leq \frac{1}{\theta}\Ent_{P_{t-s}} (h) + \frac{1}{\theta} \log P_{t-s} \left( e^{\theta \Gamma(f_s)} \right) 
$$
where $\Ent_{P_{t}}(f^2) \coloneqq P_t(f^2 \log f^2) - P_t(f^2) \log P_t(f^2)$ denotes the entropy of $f^2$ with respect to $P_t$.
Now the curvature condition implies (and in fact is equivalent to) the following local log-Sobolev inequality (see \textit{e.g.}\ \cite[Theorem 5.4.7]{ane})
\begin{equation} \label{eq:locallsob} 
\Ent_{P_t}(f^2)  \leq \frac{2}{\rho}(1-e^{-2\rho t}) P_t(\Gamma(f)) .
\end{equation}
At time $t-s$ with $f=\sqrt{h}$,  \eqref{eq:locallsob}  reads 
$$
\Ent_{P_{t-s}} (h)  \leq \frac{1-e^{-2\rho(t-s)}}{2\rho} P_{t-s}(h \Gamma( f_s )) .
$$
Therefore, for any $\theta > \frac{1-e^{-2\rho(t-s)}}{2\rho}$ it holds
\begin{align*}
P_{t-s} ( h \Gamma(f_s) ) 
& \leq 
\frac{2\rho}{2\rho \theta-1+e^{-2\rho(t-s)}}\log P_{t-s} \left( e^{\theta \Gamma(f_s)} \right) \\
& \leq 
\frac{2\rho}{2\rho\theta-1+e^{-2\rho(t-s)}}\log P_{t-s} \left( e^{\theta e^{-2\rho s} P_s(\Gamma(f))} \right) \\
& =
\frac{2\rho q(s)}{2\rho\theta-1+e^{-2\rho(t-s)}}\log \left( P_{t-s} \left( e^{q(s) P_s(\Gamma(f))} \right) \right)^{1/q(s)}
\end{align*}
where for the second inequality we used the commutation property \eqref{eq:commutaiton} and in the last equality we choose $\theta$ so that 
$$
\theta e^{-2\rho s} = q(s) \coloneqq \frac{\alpha(e^{2\rho t}-1)}{e^{2\rho(t-s)}-1} .
$$
This choice is licit as soon as 
$\alpha > \frac{e^{2\rho t} + e^{-2\rho t} -2}{2\rho(e^{2\rho t}-1)}$ (which guaranties that 
$\theta > \frac{1-e^{-2\rho(t-s)})}{2\rho}$). Now by Lemma \ref{lem:hypercontractivity}, applied at time $0$ and $s$, with $p=\alpha$, we get
\begin{align*}
P_{t-s} ( h \Gamma(f_s) ) 
& \leq 
\frac{2\rho q(s)}{2\rho\theta-1+e^{-2\rho(t-s)}}\log \left( P_{t} \left( e^{q(0) (\Gamma(f))} \right) \right)^{1/q(0)} \\
& = 
\frac{2\rho q(s)/q(0)}{2\rho\theta-1+e^{-2\rho(t-s)}}\log  P_{t} \left( e^{\alpha  (\Gamma(f))} \right)  .
\end{align*}
Using the explicit expressions of $\theta$ and $q(s)$ in terms of $s,t,\rho$, we conclude after some algebra that
$$
\log P_t \left( e^f \right) - P_t f \leq 
\log P_t \left( e^{\alpha \Gamma(f)} \right) \int_0^t \frac{2\rho (e^{2\rho t}-1)}{2\rho \alpha e^{2\rho s} (e^{2\rho t}-1) + (e^{2\rho(t-s)}-1)(e^{-2\rho(t-s)}-1)} ds.
$$
The expected result follows from Lemma \ref{lem:integral}. The proof is complete.
\end{proof}

\begin{remark}
In the course of the proof, we could have used a cruder but shorter argument. Indeed, by Jensen's inequality
$$
P_{t-s} \left( e^{\theta e^{-2\rho s} P_s(\Gamma(f))} \right) 
\leq 
P_{t-s} P_s \left( e^{\theta e^{-2 \rho s}\Gamma(f)} \right)
= 
P_t \left( e^{\theta e^{-2\rho s}\Gamma(f)} \right).
$$
Therefore, for $\theta = \theta(s) = \alpha e^{2\rho s}$, we obtain (for $\alpha > (1-e^{-2t})/2\rho$)
\begin{align*}
\log P_t (e^f) - P_t f  
& \leq 
\log P_{t} \left( e^{\alpha\Gamma(f)} \right) \int_0^t  \frac{2\rho}{2\rho \theta(s)-1+e^{-2\rho (t-s)}} ds \\
& =
\log P_{t} \left( e^{\alpha\Gamma(f)} \right)  \log \left( \frac{2\rho \alpha e^{2\rho t}}{2\rho \alpha e^{2\rho t}+1} \times \frac{2\rho \alpha+e^{-2\rho t}}{2\rho \alpha - 1 +e^{-2\rho t}}  \right)
\end{align*}
(where the equality follows by computing explicitly the integral (changing variables $u=e^{2\rho s}(2\rho\alpha+e^{-2\rho t})-1$ (the denominator)).
In the limit $t \to \infty$ the latter implies
$$
\int e^f d\mu \leq \left( \int e^{\alpha\Gamma(f)} d\mu \right)^{\log \left( \frac{2\rho \alpha }{2\rho \alpha -1} \right)}
$$
for any $\alpha > \frac{1}{2\rho}$
which already constitutes  an improvement with respect to \eqref{eq:bg}, under the curvature condition.
\end{remark}

In the proof of Theorem \ref{th:BG} we used the following lemmas.

\begin{lem} \label{lem:hypercontractivity}
Let $L$ be a Markov diffusion operator satisfying the curvature condition $\Gamma_2(f) \geq \rho \Gamma(f)$ for some $\rho >0$ and all $f \in \mathcal{A}$.  Let $0\leq s \leq s' < t$. Then  any $f \in \mathcal{A}$ non-negative satisfies
$$
(P_{t-s'}( e^{q(s')P_{s'}f}) )^\frac{1}{q(s')}
\leq 
(P_{t-s}( e^{q(s)P_sf}) )^\frac{1}{q(s)}
$$
where 
$$
q(s)=\frac{p(e^{2\rho t} -1)}{e^{2\rho(t-s)}-1} 
$$
with $p>0$.
\end{lem}

\begin{remark}
Lemma \ref{lem:hypercontractivity} corresponds to some local hypercontractivity property that is usually  stated for the invariant measure $\mu$ in place of $P_{t}$, see \cite{bakry,ane}. Some similar local version of the hypercontractivity property already exist in \cite[Theorem 3.1]{BBG}.  
It is plain and somehow classical that
the statement of Lemma \ref{lem:hypercontractivity}  is  equivalent to what appears in \cite[Theorem 3.1]{BBG}\footnote{and in fact the conclusion of the Lemma implies the curvature condition so that both properties are equivalent} but we need the specific  form of the statement above that we could not find or directly derive from \cite{BBG}.
\end{remark}

\begin{proof}
Fix $t$ and set 
$\psi(s) = \frac{1}{q(s)} \log \left( P_{t-s} \left( e^{q(s)P_s f} \right) \right)$.
Taking the derivative, we get, after some algebra
(using the diffusion property)
\begin{align*}
\psi'(s) 
&= 
-\frac{q'(s)}{q(s)^2} \log \left( P_{t-s} \left( e^{q(s)P_s f} \right) \right)
+
\frac{P_{t-s} \left( - Le^{qP_sf} + e^{qP_sf}[q'P_s f + q L P_sf] \right)}{q(s) P_{t-s} \left( e^{q(s)P_s f} \right)}
    \\
& =
\frac{q'(s)}{q(s)^2} \frac{1}{P_{t-s} \left( e^{q(s)P_s f} \right)}
\left( 
\Ent_{P_{t-s}}\left( e^{qP_sf}\right) - \frac{4q}{q'} P_{t-s} \left( \Gamma( e^{qP_sf/2})\right)
\right)
.
\end{align*}
Our choice of $q$ insures that
$$
\frac{4q(s)}{q'(s)} = \frac{2}{\rho}(1-e^{-2\rho(t-s)})
$$
so that, applying the local log-Sobolev inequality \eqref{eq:locallsob} at time $t-s$ to $f^2=e^{qP_sf}$, we conclude that $\psi'(s) \leq 0$. The expected result immediately follows.
\end{proof}

\begin{lem} \label{lem:integral}
For any $t>0$ it holds
\begin{align*}
& \int_0^t \frac{2\rho (e^{2\rho t}-1)}{2\rho \alpha e^{2\rho s} (e^{2\rho t}-1) + (e^{2\rho(t-s)}-1)(e^{-2\rho(t-s)}-1)} ds \\
& \phantom{AAAAAAAA} =
\frac{\sqrt{1-e^{-2\rho t}}}{2\sqrt{2\rho \alpha}}
\log \left( \frac{(e^{\rho t}a_t-1)(e^{\rho t}a_t+1+a_t^2-e^{-2\rho t})}
{(e^{\rho t}a_t+1)(-e^{\rho t}a_t+1+a_t^2-e^{-2\rho t})} \right) 
\end{align*}
with 
$$
a_t^2 \coloneqq 2\rho \alpha (e^{2\rho t}-1) .
$$
\end{lem}

\begin{proof}
Changing variables $u=e^{2\rho s}$, we observe that
\begin{align*}
I_t 
& \coloneqq 
\int_0^t \frac{2\rho (e^{2\rho t}-1)}{2\rho \alpha e^{2\rho s} (e^{2\rho t}-1) + (e^{2\rho(t-s)}-1)(e^{-2\rho(t-s)}-1)} ds    \\
& =
\int_1^{e^{2\rho t}} \frac{e^{2\rho t}-1}
{u \left[2\rho \alpha u(e^{2\rho t}-1) + 2 - \frac{e^{2\rho t}}{u} - e^{-2\rho t} u\right]} du \\
& =
\int_1^{e^{2\rho t}} \frac{e^{2\rho t}-1}
{u^2 c_t  + 2u - e^{2\rho t}} du
\end{align*}
where we set for simplicity 
$c_t \coloneqq 2\rho \alpha (e^{2\rho t}-1) -e^{-2\rho t}$. Denote $d_t^2 \coloneqq\frac{1}{c_t^2}+ \frac{e^{2\rho t}}{c_t}$ so that
$$
u^2 c_t  + 2u - e^{2\rho t}
=c_t \left( \left(u+\frac{1}{c_t} \right)^2 - d_t^2 \right) .
$$
Changing again variables $v=u+\frac{1}{c_t}$, it follows that 
\begin{align*}
I_t 
& = 
\frac{e^{2\rho t}-1}{c_t}\int_{1+\frac{1}{c_t}}^{e^{2\rho t}+\frac{1}{c_t}} \frac{dv}{v^2-d_t^2} 
 =
\frac{e^{2\rho t}-1}{c_t} \left[ \frac{1}{2d_t}\log \left( \frac{v-d_t}{v+d_t}\right) \right]_{1+\frac{1}{c_t}}^{e^{2\rho t}+\frac{1}{c_t}} .
\end{align*}
This leads to the expected result after some algebra.
\end{proof}


\subsection{Extension to Sub-Gaussian measures}

As already mentioned, the main assumption in the original approach of 
the inequality \eqref{eq:bg} in \cite{bobkov-gotze} is the log-Sobolev inequality. 
For the one parameter family of measures $\mu_p$ with density $Z_p^{-1}e^{-|x|^p}$ 
on the line, say, with $Z=\int_\mathbb{R} e^{-|x|^p} dx$, it is known that 
a log-Sobolev inequality holds if and only if $p \geq 2$.

In this section we make use of a different functional inequality in order 
to derive some exponential integrability bounds in the flavour of 
\eqref{eq:bg} but for the measure $\mu_p$, 
$1 < p \leq 2$. Our approach is an extension of the proof of \cite{bobkov-gotze}. In particular the semi-group approach of Theorem \ref{th:BG} cannot be applied as it is since we would need a commutation property similar to \eqref{eq:commutaiton} which does not hold for the diffusion semi-group associated to the measure $\mu_p$, $p \in (1,2)$ (the curvature condition only reads $\Gamma_2 \geq 0$).

We work on $\mathbb{R}^n$ but our result should extend to general diffusion as in Section \ref{sec:gamma2}. The choice of $\mathbb{R}^n$ is motivated by the fact that, to the best of our knowledge, the modified log-Sobolev inequality we are about to introduce has been studied and proved only in $\mathbb{R}^n$ (see \cite{GGM05,GGM07,barthe-roberto-08}).


A probability measure $\mu$ on $\mathbb{R}^n$ satisfies the modified log-Sobolev 
inequality with function $H \colon \mathbb{R} \to [0,\infty)$ if there exists a 
constant $c \in (0,\infty)$ such that for all locally Lipschitz function 
$f \colon \mathbb{R}^n \to \mathbb{R}$ the following holds
\begin{equation*} 
\Ent_\mu(f^2) \leq c \int \sum_{i=1}^n H \left( \frac{\partial_i f}{f} \right)f^2 d\mu 
\end{equation*}
where $\partial_i f = \frac{\partial}{\partial x_i}f$ and 
$$
\Ent_{\mu}(f)
:=
\int f \log f d\mu - \int fd\mu \log \int fd\mu 
=
\sup\left\{\int fgd\mu:\int e^{g}d\mu \leq 1\right\} .
$$


Similar to the log-Sobolev inequality, this is equivalent to
\begin{equation} \label{eq:modifiedlsob}
\Ent_\mu(e^f) \leq c \int \sum_{i=1}^n H \left( \frac{\partial_i f}{2} \right)e^f d\mu ,
\end{equation}

for all locally Lipschitz functions $f$.
The terminology goes back to \cite{bobkov-ledoux-97} in their study of the concentration phenomenon for product of exponential measures, thus recovering a celebrated result by Talagrand \cite{talagrand-91}.
The second author and Ledoux  introduced and considered  \eqref{eq:modifiedlsob} for functions such that $\partial_i f/f \leq \kappa <1$ which amounts to taking $H(x) = 2x^2/(1 - \kappa)$ if $|x| \leq \kappa$ and $H(x) = \infty$ otherwise.

The same authors proved \eqref{eq:modifiedlsob} for the family of measures $\mu_p$, with $p>2$ (they consider also measures of the form $e^{-V}$ with $V$ strictly uniformly convex), and $H(x)=c_p|x|^q$ with $q = \frac{p-1}{p} \in [1,2]$ the conjugate of $p$. Such inequalities are equivalent to $q$-log-Sobolev inequalities that are studied in depth in \cite{bobkov-zegarlinski}.

The case $p \in (1,2)$  was considered by Gentil, Guillin and Miclo \cite{GGM05} who
established a modiﬁed log-Sobolev inequality for $\mu_p$ when $p \in (1, 2)$, with  $H$ that compares to $\max(x^2,|x|^q)$, $q = \frac{p-1}{p} > 2$.
The same authors extended their result to a wider class of 
log-concave measures on the line, with tail behavior between exponential and Gaussian \cite{GGM07}. Another approach, based on Hardy-type inequalities, was proposed in \cite{barthe-roberto-08}. 
For more results related to the measures $\mu_p$, $p \in (1,2)$, we refer the reader 
to \cite{CPS15,RZ}. 

We are now in position to state our theorem.

\begin{thm} \label{th:bobkov-gotze-modified}
Assume that the probability measure $\mu$ on $\mathbb{R}^n$ satisfies the modified log-Sobolev inequality \eqref{eq:modifiedlsob} with a non decreasing convex function $H \colon \mathbb{R} \to [0,\infty)$ vanishing at the origin (and constant $c$). Assume furthermore that $x \mapsto H(\sqrt{x})$ is also convex.
Then, for all locally Lipschitz function $f \colon \mathbb{R}^n \to \mathbb{R}$
with $\int f d\mu =0$, the following inequality holds
$$
\int e^f d\mu \leq \left( \int \exp \left\{ \alpha \sum_{i=1}^n H \left(\frac{\partial_i f}{2} \right) \right\} d\mu \right)^\frac{c}{\alpha-c}
$$
for all $\alpha > c$.
\end{thm}

\begin{remark}
{\rm The classical log-Sobolev inequality corresponds to $H(x)=x^2$.  We therefore recover the result of \cite{bobkov-gotze}
(for $\mathbb{R}^n$).

For the measure $\mu_p$, $p \in (1,2)$ mentioned above, one can prove \eqref{eq:modifiedlsob} \cite{GGM05,barthe-roberto-08} with $H$ that compares to $\max(|x|^q,x^2)$, with $q = \frac{p-1}{p} > 2$ the conjugate of $p$. 
At the price of changing the constant $c$, it is easy to construct an other $H$ that satisfies the assumption of the theorem for $\mu_p$ for which the conclusion applies.

By mean of \eqref{eq:modifiedlsob} for  more general log-concave measures with tail between exponential and Gaussian \cite{GGM07,barthe-roberto-08} the conclusion of the theorem can also be extended to such a class of
probability measures.

It should be noticed that, except for the Gaussian measure, the optimal constant in Inequality \eqref{eq:modifiedlsob} is not known.   Therefore the constraint $\alpha > c$ might be sub-optimal. }
\end{remark}

\begin{proof}
We follow \cite{bobkov-gotze}. Fix $\lambda, \alpha \geq 0$. 
Using the variational characterization of the entropy 
$\Ent_\mu(e^{\lambda f}) = \sup \{ \int e^{\lambda f} gd\mu : \int e^g d\mu =1\}$, 
we get, for 
$$
g= \alpha \sum_{i=1}^n H \left( \frac{\lambda}{2} \partial_i f \right)
- \log \int \exp\left\{ \alpha \sum_{i=1}^n H \left( \frac{\lambda}{2} \partial_i f \right) \right\} d\mu ,
$$
$$
\Ent_\mu(e^{\lambda f}) 
\geq 
\alpha \int  e^{\lambda f} \sum_{i=1}^n H \left( \frac{\lambda}{2} \partial_i f \right) d\mu 
- \int e^{\lambda f} d\mu \log \int \exp \left\{ \alpha \sum_{i=1}^n H \left( \frac{\lambda}{2} \partial_i f \right) \right\} d\mu .
$$
On the other hand, the modified log-Sobolev inequality \eqref{eq:modifiedlsob} applied to $\lambda f$ ensures that
$$
\alpha \int  e^{\lambda f} \sum_{i=1}^n H \left( \frac{\lambda}{2} \partial_i f \right) d\mu 
\geq 
\frac{\alpha}{c}\Ent_\mu(e^{\lambda f}) 
.
$$
Therefore, for $\alpha > c$,
$$
\Ent_\mu(e^{\lambda f})
\leq 
\frac{c}{\alpha -c}
\int e^{\lambda f} d\mu \log \int \exp \left\{ \alpha \sum_{i=1}^n H \left( \frac{\lambda}{2} \partial_i f \right) \right\} d\mu
$$
The rest of the argument is Herbst's argument. Set $G(\lambda)\coloneqq \int e^{\lambda f}d\mu$ and observe that
$$
\Ent_\mu(e^{\lambda f}) = \lambda G'(\lambda) - G(\lambda) \log G(\lambda) 
$$
so that the inequality above can be recast as
$$
\frac{G'(\lambda)}{\lambda G(\lambda)} - \frac{\log(G(\lambda))}{\lambda^2}
= \frac{d}{d\lambda}\left( \frac{1}{\lambda} \log G (\lambda) \right) 
\leq 
\frac{c}{\alpha -c}
\frac{1}{\lambda^2} \log \int \exp \left\{ \alpha \sum_{i=1}^n H \left( \frac{\lambda}{2} \partial_i f \right) \right\} d\mu .
$$
Now introduce, for $x \geq 0$,
$\omega(x) = \sup_{t >0} \frac{H(tx)}{H(t)}$ hence it follows that:
$$
\frac{d}{d\lambda}\left( \frac{1}{\lambda} \log G (\lambda) \right) 
\leq 
\frac{c}{\alpha -c} \frac{\beta(\lambda^2)}{\lambda^2}
$$
with 
$$
\beta(\lambda) \coloneqq 
\log \int \exp \left\{ \alpha \omega(\sqrt{\lambda}) \sum_{i=1}^n H \left( \frac{\partial_i f}{2} \right) \right\} d\mu .
$$
Since $H(\sqrt{x})$ is convex, it is easy to check that $\omega(\sqrt{x})$ is convex and therefore that $\beta$ is also convex. Since $\omega(0)=0$ ($H$ vanishes at $0$), the map $t \to \beta(t)/t$
is non-decreasing and therefore 
$$
\frac{\beta(\lambda^2)}{\lambda^2} \leq \beta(1) =
\log \int \exp \left\{ \alpha \omega(1) \sum_{i=1}^n H \left( \frac{\partial_i f}{2} \right) \right\} d\mu
=
\log \int \exp \left\{ \alpha \sum_{i=1}^n H \left( \frac{\partial_i f}{2} \right) \right\} d\mu
$$
since $\omega(1)=1$. Integrating the differential inequality, we obtain
$$
\log G(1) - \lim_{\lambda \to 0} \frac{1}{\lambda} \log G(\lambda) \leq \frac{c}{\alpha -c } \beta(1)
$$
which leads to the desired conclusion since 
$\lim_{\lambda \to 0} \frac{\log G(\lambda)}{\lambda}  = \int f d\mu=0$.
\end{proof}

\subsection{Exponential integrability of second order}

In \cite{CMP21} the authors are interested in the following second order type exponential integrability. They prove that, given $\beta \geq 1$ and  $M>0$, all  $f$ smooth enough, with $\int f d\gamma_n =0$, and such that  $\int e^{(\frac{\beta}{2}|Lf|)^\beta} d\gamma_n \leq M$ satisfy
$$
\int e^{|f|^\beta} d\gamma_n \leq C(M)
$$
for some constant $C(M)$ (independent of $f$), where $L = \Delta - x \cdot \nabla$ is the Ornstein-Uhlenbeck operator. 
In this section, we give a quantitative version of this statement in the special case $\beta=1$. Our result extends \cite{CMP21} (for $\beta=1$) to the class of measures on $\mathbb{R}^n$ that satisfy the log-Sobolev Inequality
(this class is larger that the set of all strongly log-concave probability measures considered in \cite{CMP21}) and to the general framework allowed by $\Gamma_2$ formalism. Our result  reads as follows.

\begin{thm} \label{thm:CMP-Lf}
Let $L$ be   some Markov diffusion operator reversible with respect to some probability measure $\mu$. Assume that $\mu$ satisfies the following log-Sobolev inequality 
\begin{equation} \label{eq:log-sob}
    \Ent_\mu(e^f) \leq c \int e^f \Gamma(f) d\mu
\end{equation}
for some $c \in (0,\infty)$ and any $f$ smooth enough.
Then, for any $\alpha > c$ and any $f$ with $\int fd\mu=0,$
the following inequality holds
$$
\log  \int e^f d\mu  
\leq \frac{c}{\alpha -c} \int e^{\alpha |Lf|} d\mu .
$$
Moreover we have:
$$
\log \int e^{|f|} d\mu 
\leq 
\left( \frac{c}{e \alpha} + \log 2 +  \frac{2c}{\alpha -c} \right)
 \int e^{\alpha |Lf|} d\mu .
$$
\end{thm}

\begin{remark}
For the Gaussian measure $c=1/2$. The constraint $\alpha > 1/2$ agrees 
with the optimal threshold $\beta/2=1/2$ of \cite{CMP21}.
\end{remark}

\begin{proof}[Proof of Theorem \ref{thm:CMP-Lf}]
By the diffusion property, $\Gamma(e^f,f)=e^f \Gamma(f)$ so that the integration by parts formula ensures that
$\int e^f \Gamma(f) d\mu = \int \Gamma(e^f,f) d\mu = -\int e^f Lf d\mu$.
In particular, the log-Sobolev inequality can be recast as
$$
\Ent_\mu(e^f) \leq c \int e^f (-Lf) d\mu .
$$
Next, if $f$ is such that $\int e^fd\mu=1$, the entropic inequality ensures 
that, for any $t >0$,
$$
\int e^f (-Lf) d\mu 
\leq 
\frac{1}{t} \Ent_\mu(e^f) + \frac{1}{t} \log \left( \int e^{-tLf} d\mu \right) .
$$
It follows that for any $t > c$ 
$$
\Ent_\mu(e^f) \leq 
\frac{c}{t-c} \log \left( \int e^{-tLf} d\mu \right) .
$$
In turn, applying this to $\lambda f$, the following inequality holds by homogeneity
$$
\Ent_\mu(e^{\lambda f}) \leq 
\frac{c}{t-c} \int e^{\lambda f}d\mu \log \left( \int e^{-t\lambda Lf} d\mu \right), \qquad \lambda \in [0,1] .
$$
Set as usual $H(\lambda)=\int e^{\lambda f}d\mu$. Then the latter reads
$$
\lambda H'(\lambda) - H(\lambda) \log H(\lambda) \leq 
\frac{c}{t-c} H(\lambda)  \log \left( \int e^{-t\lambda Lf} d\mu \right) . 
$$
Dividing both sides by $\lambda^2H(\lambda)$ as in the usual Herbst's argument, we obtain
$$
\frac{d}{d\lambda} \left( \frac{\log H(\lambda)}{\lambda} \right) = 
\frac{H'(\lambda)}{\lambda H(\lambda)} - \frac{\log H(\lambda)}{\lambda^2} 
\leq 
\frac{c}{t-c}  \frac{1}{\lambda^2}\log \left( \int e^{-t \lambda Lf} d\mu \right) , \qquad t > c. 
$$
Jensen's inequality $\log \left( \int e^{-t\lambda Lf} d\mu \right) \leq \lambda \log \left( \int e^{-t Lf} d\mu \right)$ would not be enough to make the right hand side integrable at $0$ (with respect to $\lambda$). 
To overcome this lack of integrability, we will expand the exponential, taking advantage of the fact that $\int Lf d\mu=0$ by reversibility (thus making the right hand side of order $\int (Lf)^2d\mu$ when $\lambda \to 0$). More precisely, we have for any $\lambda \in (0,1]$
\begin{align*}
\frac{1}{\lambda^2} \log \left( \int e^{-t \lambda  Lf} d\mu \right) & = 
\frac{1}{\lambda^2} \log \left( 1 + \sum_{k=2}^\infty \frac{(-t \lambda)^k}{k!}\int (Lf)^k  d\mu \right)  & (\mbox{since } \int Lfd\mu=0)\\
& \leq 
\sum_{k=2}^\infty \frac{t^k \lambda^{k-2}}{k!}\int |Lf|^k  d\mu & (\mbox{since } \log(1+x) \leq x) \\
& \leq 
\int e^{t |Lf|} d\mu .
\end{align*}
Finally, we get
$$
\frac{d}{d\lambda} \left( \frac{\log H(\lambda)}{\lambda} \right) 
\leq 
\frac{c}{t-c}  \int e^{t |Lf|} d\mu
$$
which leads to the first desired conclusion after integration from $\lambda=0$ to $\lambda=1$ since 
$\lim_{\lambda \to 0} \frac{\log H(\lambda)}{\lambda} = \int f d\mu$.

The second part of the theorem is a consequence of the first.
 Given $f$ with $\int f d\mu=0$, set $f_+=\max(f,0)$ and $f_-=\max(-f,0)$. Then,
since $\log(a+b) \leq \log a + \log b + \log 2$ for all $a,b \geq 1$, the first part of the theorem applied twice by scaling argument implies that
\begin{align*}
\log \int e^{|f|} d\mu 
& \leq 
\log \left( \int e^{f_+} d\mu +  \int e^{f_-} d\mu \right) \\
& \leq 
\log \left( \int e^{f_+} d\mu \right)
+
  \log \left(\int e^{f_-} d\mu \right) + \log 2 
\\
& \leq 
\int (f_+ + f_-) d\mu + \log 2 +  \frac{c}{\alpha -c} \left( \int e^{\alpha |Lf_+|} d\mu + \int e^{\alpha |Lf_-|} d\mu \right) \\
& \leq 
\int |f| d\mu + \log 2 +  \frac{2c}{\alpha -c} \int e^{\alpha |Lf|} d\mu .
\end{align*}
Since $\mu$ satisfies the log-Sobolev inequality \eqref{eq:log-sob}, it satisfies a Poincar\'e inequality with constant $c/2$ (see \textit{e.g.}\ \cite[Chapter 1]{ane}). In particular (recall that $\int fd\mu=0$)
$$
\int f^2 d\mu \leq \frac{c}{2} \int \Gamma(f) d\mu .
$$
Since
$\int \Gamma(f) d\mu = - \int f Lf d\mu$, the Cauchy-Schwartz Inequality implies
$$
\int f^2 d\mu \leq
\frac{c}{2} \sqrt{\int f^2 d\mu \int (Lf)^2 d\mu} 
$$
from which we deduce that
$$
\int f^2 d\mu \leq
\frac{c^2}{4} \int (Lf)^2 d\mu .
$$
Using again Cauchy-Schwartz' inequality, we obtain
$$
\int |f| d\mu \leq \sqrt{ \int f^2 d\mu  }
\leq 
\frac{c}{2} \sqrt{\int (Lf)^2 d\mu} .
$$
Now we observe that $x^2 \leq \frac{4}{e^2}e^{x}$ for $x \geq 0$, hence
$\int \alpha^2 (Lf)^2 d\mu \leq \frac{4}{e^2} \int e^{\alpha |Lf|} d\mu$ and therefore
$$
\int |f| d\mu
\leq 
\frac{c}{2\alpha} \sqrt{\frac{4}{e^2} \int e^{\alpha |Lf|} d\mu}
\leq 
\frac{c}{e \alpha} \int e^{\alpha |Lf|} d\mu , \qquad \alpha >0 .
$$
We get for any $\alpha > c$
\begin{align*}
\log \int e^{|f|} d\mu 
& \leq 
\frac{c}{e \alpha} \int e^{\alpha |Lf|} d\mu +
\log 2 
+
 \frac{2c}{\alpha -c} \int e^{\alpha |Lf|} d\mu \\
& \leq 
 \left( \frac{c}{e \alpha} + \log 2 +  \frac{2c}{\alpha -c} \right)
 \int e^{\alpha |Lf|} d\mu .
\end{align*}
The desired conclusion follows.
\end{proof}


\section{A convexity argument}

In this section we will give an alternative proof of Inequality \eqref{eq:IR} (in a slight modified form). Our proof is based on an argument of duality that applies to part of the more general situation considered in \cite{CMP}.

\subsection{Inequality \eqref{eq:IR} revisited}

Set for simplicity $\gamma=\gamma_1$ for the standard Gaussian measure on the line, with density $e^{-x^2/2}/\sqrt{2\pi}$. Our next result is one of our main results and is, as already mentioned in Section \ref{sec:first}, some sort of exponential Hardy-type inequality.

\begin{thm} \label{th:exp-hardy}
For all $f \colon [0,\infty) \to \mathbb{R}$ with $f(0)=0$, it holds that
$$
\log \int_0^\infty e^{|f(x)|} e^{-\frac{x^2}{2}} dx 
\leq  
\int_0^\infty \frac{e^{|f'(x)|^2/2}}{\sqrt{1+\frac{|f'(x)|^2}{2}}} e^{-\frac{x^2}{2}} dx
+
5.14  .
$$
\end{thm}

Before moving to the proof of Theorem \ref{th:exp-hardy}, let us show how to recover Inequality \eqref{eq:IR} from it. 

\begin{coro} \label{cor:IR}
For any dimension $n$ and any $f$ with $\int f d\gamma_n =0$, the following inequality holds
$$
\log \int e^f d\gamma_n 
\leq 
8 \int_{\mathbb{R}^n} \frac{e^{\frac{|\nabla f|^2}{2}}}{ \sqrt{1+ \frac{|\nabla f|^2}{2}}} d\gamma_n 
\leq 
14 \int_{\mathbb{R}^n} \frac{e^{\frac{|\nabla f|^2}{2}}}{1+ |\nabla f|} d\gamma_n.
$$
\end{coro}

\begin{proof}[Proof of Corollary \ref{cor:IR}]
The exponential Hardy-type inequality of Theorem \ref{th:exp-hardy}
implies that 
$$
\int_0^\infty e^f d\gamma - \frac{1}{2}
\leq 
F \left( \int G(|f'|) d\gamma \right) , \qquad f(0)=0
$$
with  $F(x)=ae^{bx} - c$, $a=\frac{e^{5.14}}{\sqrt{2\pi}}$, $b= \sqrt{2\pi}$, $c=\frac{1}{2}$
and $G(x)=e^{x^2/2}/\sqrt{1+(x^2/2)}$. Therefore the assumptions of Lemma \ref{lem:paris} are satisfied. Since 
$\max_{x \geq 0} \frac{x}{G(x)}=
\sqrt{1+\sqrt{2}}e^{-\frac{1}{\sqrt{2}}}$
(the maximum is reached at $x = 2^{1/4}$), 
$$
d
\coloneqq
\sqrt{\frac{\pi}{2}}
\max_{x \geq 0} \frac{x}{G(x)}
\simeq 0.9602
\leq 1
$$
and $d+b \simeq 3.467 \leq 3.5$.
Lemma \ref{lem:paris} implies that, for any dimension $n$ and any $f$ with $\int fd\gamma_n = 0$, it holds
\begin{align*}
\log \int e^g d\gamma_n 
& \leq 
\log \left(1 + a \exp \left\{3.5\int G(|\nabla f|) d\gamma_n \right\}
+(a-1) \exp \left\{d\int G(|\nabla f|) d\gamma_n \right\} \right)
\end{align*}
It is not difficult to check that
$x \mapsto \Psi(x):= \frac{1}{\log(x)} \log \left(1 + a x^{3.5} + (a-1)x \right)$ is decreasing on $[e,\infty)$ so that, for $x=\exp \left\{\int G(|\nabla f|) d\gamma_n \right\}$,
\begin{align*}
\log \int e^f d\gamma_n 
& \leq 
\Psi(e)  \int_{\mathbb{R}^n} \frac{e^{\frac{|\nabla f|^2}{2}}}{\sqrt{1+ \frac{|\nabla f|^2}{2}}} d\gamma_n 
\leq 8
 \int_{\mathbb{R}^n} \frac{e^{\frac{|\nabla f|^2}{2}}}{\sqrt{1+ \frac{|\nabla f|^2}{2}}} d\gamma_n 
\end{align*}
where we used that $\Psi(e)   \leq 8$ (by numerical computation).
To conclude, it is enough to observe that $\frac{1}{\sqrt{1+\frac{x^2}{2}}} \leq \frac{\sqrt{3}}{1+x}$ (the maximum is reached at $x=2$), and that $8\sqrt{3} \leq 14$.
\end{proof}

\begin{proof}[Proof of Theorem \ref{th:exp-hardy}]
Since $f(0)=0$, for $x >0$, we have by Young's inequality, 
\begin{align*}
|f(x)| 
& = 
| \int_0^x f'(t) e^{t^2/2} e^{-t^2/2} dt |    \\
& \leq 
 \int_0^x G(|f'(t)|) e^{-t^2/2} dt + \int_0^x G^*(e^{t^2/2}) e^{-t^2/2} dt 
\end{align*}
for any $G \colon [0,\infty) \to \mathbb{R}$ and $G^*(y)=\sup_{x>0} \{xy - G(x)\}$, $y \geq 1$.
It follows that
$$
\log \int_0^\infty e^{|f(x)|} e^{-\frac{x^2}{2}} dx \leq  
\int_0^\infty G(|f'(x)|) e^{-\frac{x^2}{2}} dx
+
\log \int_0^\infty e^{\int_0^x \left[ G^*(e^{t^2/2}) e^{-t^2/2} - t \right]dt}  dx .
$$
To end the proof, it remains to show that the second term in the right hand side of the latter is bounded above. Having a non-explicit constant would just necessitate the asymptotic of $G^*$ for 
 the choice $G(x)=e^{V(x)}$ with $V(x) \coloneqq \frac{x^2}{2}-\frac{1}{2}\log (1+\frac{x^2}{2})$,  $x \geq 0$. This could be achieved in few lines. We choose however to make some effort to keep track of the constants, making the presentation much lengthier. In other words the proof of Theorem \ref{th:exp-hardy} is essentially finished, what remains to be done is fixing some heavy technicalities to get an explicit constant in the end.

Observe that $V'(x)=\frac{x(1+x^2)}{2+x^2}$ and set 
$$
W(x)=V(x)+ \log V'(x) = \frac{x^2}{2} + \log H(x),
\qquad H(x) = \frac{\sqrt{2}x(1+x^2)}{(2+x^2)^{3/2}} , \quad x >0 .
$$
The maps $W, H$ are increasing from $(0,\infty)$ to $(-\infty,\infty)$ and from 
$(0,\infty)$ to $(0,\sqrt{2})$ respectively, so that we can safely define their inverse $W^{-1},H^{-1}$ (see Lemma \ref{lem:technical}). Set $A_o=H^{-1}(1)$ and $x_o=W^{-1}(0)$. We observe that
$$
G^*(y) = y \left( W^{-1}(\log y) - \frac{1}{V'(W^{-1}(\log y))}\right), \qquad y \geq 1
$$
and (for $t \geq 0$)
\begin{align*}
G^*(e^{t^2/2}) e^{-t^2/2} - t 
& = 
W^{-1}(t^2/2) - \frac{1}{V'(W^{-1}(t^2/2))} - t \\
& =
x - \frac{1}{V'(x)} - \sqrt{2W(x)}
\end{align*}
where we have set $x=W^{-1}(t^2/2)$. Denote $t_o = \sqrt{2W(4)} \simeq 4.057$ and observe that, by monotonicity of $W$, there is a correspondence between the notations $[0, t_o]$ and $[x_o, 4]$ through the change of variables $x=W^{-1}(t^2/2)$). Therefore, thanks to Item $(iii)$ of Lemma \ref{lem:technical} that guarantees that the above quantity is non-positive,  the following holds
\begin{align*}
\int_0^\infty e^{\int_0^x \left[ G^*(e^{t^2/2}) e^{-t^2/2} - t \right]dt}dx 
& = 
\int_0^{t_o} e^{\int_0^x \left[ G^*(e^{t^2/2}) e^{-t^2/2} - t \right]dt} dx 
+ 
\int_{t_o}^\infty e^{\int_0^x \left[ G^*(e^{t^2/2}) e^{-t^2/2} - t \right]dt}dx \\
& \leq 
\int_0^{t_o} dx 
+ 
\int_{t_o}^\infty e^{ \int_{t_o}^x  \left[ G^*(e^{t^2/2}) e^{-t^2/2} - t \right]dt }dx .
 \\
 & =
 t_o + 
 \int_{t_o}^\infty e^{ \int_{t_o}^x  \left[ G^*(e^{t^2/2}) e^{-t^2/2} - t \right]dt} dx .
\end{align*}
Observe that, for any $x \geq A_o$ $W(x) \geq x^2/2$ so that (composing by $W^{-1}$ that is increasing) for all $t \geq A_o$ (and therefore for all $t \geq t_o$), $W^{-1}(t^2/2) \leq t$. It follows from  Item $(iv)$ of Lemma \ref{lem:technical} that
\begin{align*}
    \int_{t_o}^x  \left[ G^*(e^{t^2/2}) e^{-t^2/2} - t \right]dt
    & \leq  
    - \int_{t_o}^x \frac{1.228}{W^{-1}(t^2/2)} dt \\
    & \leq 
    - \int_{t_o}^x \frac{1.228}{t} dt \, = \,
    - 1.228 \log \left( \frac{x}{t_o}\right) .
\end{align*}
In turn, 
\begin{align*}
    \int_{t_o}^\infty e^{ \int_{t_o}^x  \left[ G^*(e^{t^2/2}) e^{-t^2/2} - t \right]dt }dx
    & \leq 
    \int_{t_o}^\infty \left( \frac{x}{t_o}\right)^{-1.228}  dx \, = \,
    \frac{1}{0.228 t_o} .
\end{align*}
We conclude that
\begin{align*}
\int_0^\infty e^{\int_0^x \left[ G^*(e^{t^2/2}) e^{-t^2/2} - t \right]dt}dx 
\leq t_o +  \frac{1}{0.228 t_o } \simeq 5.138 \leq 5.14 
\end{align*}
as expected.
\end{proof}

The next lemma collects some technical facts about the function $V$ and $W$ appearing in the proof of the previous theorem.

\begin{lem} \label{lem:technical}
Set 
$$
V(x)=\frac{x^2}{2}-\frac{1}{2}\log (1+\frac{x^2}{2}), \quad
H(x)=\frac{\sqrt{2}x(1+x^2)}{(2+x^2)^{3/2}}, \quad x \geq 0,
$$
(observe that $V'(x)=\frac{x(1+x^2)}{2+x^2}$) and 
$$
W(x)=V(x)+ \log V'(x) = \frac{x^2}{2} + \log H(x) .
$$
Then\\
$(i)$ 
the map $H$ is increasing (one to one from $(0,\infty)$ to $(0,\sqrt{2})$)
and $A_o \coloneqq H^{-1}(1) \in (2.13,2.14)$;\\
$(ii)$ 
the map  $W$ is increasing (one to one from $(0,\infty)$  to $(-\infty,\infty)$) and $x_o \coloneqq W^{-1}(0) \in (1.05,1.06)$;\\
$(iii)$ 
for all $x \geq x_o$,
$$
x - \frac{1}{V'(x)} - \sqrt{2W(x)} \leq 0 ;
$$
$(iv)$ for all $x \geq 4$, 
$$
x - \frac{1}{V'(x)} - \sqrt{2W(x)} \leq - \frac{1.228}{x}
$$
\end{lem}

\begin{proof}
Point $(i)$ and $(ii)$ are obvious and left to the reader. By numerical computations $A_o \in (2.13, 2.14)$ and $x_o \in (1.05,1.06)$.
For $(iii)$ we observe that, for all $x \geq x_o$,
\begin{align}
x - \frac{1}{V'(x)} - \sqrt{2W(x)} 
& =
- \frac{1}{V'(x)} + \frac{x^2 - 2W(x)}{x + \sqrt{2W(x)}} \nonumber \\
& =
- \frac{1}{V'(x)} - \frac{2 \log H(x)}{x + \sqrt{2W(x)}} . \label{eq:iv}
\end{align}
For $x \geq A_o$, $H(x) \geq 1$ and therefore the above quantity is negative. 
Now assume that $x \in (x_o,A_o)$. We first deal with $x \in (x_o,2^{1/4})$. In that regime we write
\begin{align*}
    x - \frac{1}{V'(x)} - \sqrt{2W(x)}
    & \leq 
    x  - \frac{1}{V'(x)} 
    =
    \frac{x^4-2}{x(1+x^2)} \leq 0 .
\end{align*}
For $x \in (2^{1/4},A_o)$, since $H(x) \leq 1$, $V'(x) \leq x$ and $\sqrt{2W(x)} \geq 0$, the following holds
\begin{align*}
x - \frac{1}{V'(x)} - \sqrt{2W(x)} 
& =
- \frac{1}{V'(x)} - \frac{2 \log H(x)}{x + \sqrt{2W(x)}} \\
& \leq 
- \frac{1}{x} - \frac{2 \log H(x)}{x} \\
& = 
- \frac{2}{x} \log (\sqrt{e}H(x)) .
\end{align*}
By monotonicity of $H$  on $(2^{1/4},A_o)$ we have  
$- \log \left( \sqrt{e} H(x)  \right) \leq - \log (\sqrt{e} H(2^{1/4})) \leq 0$
since $\sqrt{e} H(2^{1/4}) = \sqrt{2e}2^{1/4}(1+\sqrt{2})/(2+\sqrt{2})^{3/2} \geq 1$. This proves Point $(iii)$. 

To prove the statement of Point $(iv)$, we come back to \eqref{eq:iv} and use the fact that $V'(x) \leq x$ to write
\begin{align*}
x - \frac{1}{V'(x)} - \sqrt{2W(x)} 
& \leq
- \frac{1}{x} - \frac{2 \log H(x)}{x + \sqrt{2W(x)}} \\
& =
- \frac{1+I(x)}{x}
\end{align*}
with 
$$
I(x) \coloneqq  
\frac{2x \log H(x)}{x + \sqrt{2W(x)}} 
= 
\frac{2\log H(x)}{1 + \sqrt{1 + \frac{2 \log H(x)}{x^2}}} . 
$$
Since
Since $y \mapsto \frac{y}{1+\sqrt{1+y}}$ and $H$ are increasing, we get, for any $x \geq 4$
(by numerical computations)
$$
I(x) 
\geq 
\frac{2 \log H(4)}{1 + \sqrt{1 + \frac{2\log H(4)}{x^2}}} 
\geq 
\frac{2 \log H(4)}{1 + \sqrt{1 + \frac{2\log H(4)}{4^2}}} 
\geq 0.228 .
$$
The expected result of Point $(iv)$ follows. 
\end{proof}

\subsection{Extension}

In the next theorem, we extend the previous approach to more general exponential inequalities. 
To keep the length of the paper reasonable we will not, in this section, keep track of the explicit constants.

Our starting point is the following (exponential Hardy-type inequality).

\begin{thm}
For all $\beta \in (\sqrt{5}-1,2)$, there exists a constant $c_\beta \in (0, \infty)$
such that for all $f \colon [0,\infty) \to [0,\infty)$ with $f(0)=0$ the following inequality holds
$$
\log \int_0^\infty e^{|f(x)|^\frac{2\beta}{\beta+2}} e^{-\frac{x^2}{2}} dx 
\leq  
\left(\int_0^\infty e^{ |\kappa f'(x)|^\beta} e^{-\frac{x^2}{2}} dx \right)^\frac{2\beta}{\beta+2}
+
c_\beta 
$$
with $\k= \frac{\sqrt{2}\beta}{\beta+2}$. One can take
$$
c_\beta = \frac{5}{(2-\beta)^2} - \log(\beta-\sqrt{5}+1).
$$
\end{thm}

\begin{remark}
{\rm
The constant $c_\beta = \frac{5}{(2-\beta)^2} - \log(\beta-\sqrt{5}+1)$  
explodes both at $\beta=\sqrt{5}-1$ and $\beta=2$. Indeed, for some technical 
reason the above result does not contain the case $\beta=2$ that was 
independently considered in Theorem \ref{th:exp-hardy}. It appears that many 
simplification occur when $\beta=2$ that cannot be transposed to $\beta \neq 2$.  

  Let us mention that we were not able to extend the semi-group approach developed in \cite{IR} or in Section \ref{sec:semigroup} to prove the Theorem above.

  Finally we observe that, with much more effort, it is possible to improve the right-hand side of the exponential Hardy-type inequality appearing in the theorem, by adding an extra factor in the denominator. Namely, one can replace
$e^{|\kappa f'|^\beta}$ by $e^{|\kappa f'|^\beta} / H(f')$ for some increasing function $H$ tending to infinity at infinity. }
\end{remark}

\begin{proof}
The proof follows the line of Theorem \ref{th:exp-hardy}. By Young's inequality, 
\begin{align*}
|f(x)| 
& = 
| \int_0^x f'(t) e^{t^2/2} e^{-t^2/2} dt |    \\
& \leq 
 \int_0^x G(|f'(t)|) e^{-t^2/2} dt + \int_0^x G^*(e^{t^2/2}) e^{-t^2/2} dt 
\end{align*}
for any $G \colon [0,\infty) \to \mathbb{R}$ and $G^*(y)=\sup_{x>0} \{xy - G(x)\}$, $y \geq 1$.
It follows that
\begin{align*}
& \int_0^\infty e^{|f(x)|^\frac{2\beta}{\beta+2}} e^{-\frac{x^2}{2}} dx  \leq  \\
& \int_0^x \exp \left\{ \left( 
\int_0^\infty G(|f'(t)|) e^{-\frac{t^2}{2}} dt
+
\int_0^x \left[ G^*(e^{t^2/2}) e^{-t^2/2}  \right]dt  \right)^\frac{2\beta}{\beta+2} - \frac{x^2}{2}  \right\} dx.
\end{align*}
Our aim is to analyse the second term in the right hand side of the latter. Assume that $G(x)=e^{V(x)}$ with $V(x)= (\kappa x)^\beta$,  $x \geq 0$ and $\kappa:= \frac{\sqrt{2}\beta}{\beta+2}$.
Set
$$
W(x):=V(x)+ \log V'(x) = (\kappa x)^\beta + \log (\beta \kappa^\beta x^{\beta-1}), 
\quad x >0 .
$$
The map $W$ is one to one increasing from $(0,\infty)$ to $(-\infty,\infty)$ 
so that we can define its inverse $W^{-1}$. In particular,
a direct computation shows that 
$$
G^*(y) = y \left( W^{-1}(\log y) - \frac{1}{V'(W^{-1}(\log y))}\right), \quad y \geq 1,
$$
so that (for $t \geq 0$)
$$
G^*(e^{t^2/2}) e^{-t^2/2} 
= 
W^{-1}(t^2/2) - \frac{1}{V'(W^{-1}(t^2/2))}.
$$
Set 
$$
M := \int_0^\infty G(|f'(x)|) e^{-\frac{x^2}{2}} dx = 
\int _0^\infty e^{|\kappa f'(x)|^\beta} e^{-\frac{x^2}{2}} dx 
\quad 
\mbox{and} 
\quad  
x_o := \sqrt{2W\left(1/\left(\kappa \beta^\frac{1}{\beta} \right)\right)}.
$$
Then, thanks to Lemma \ref{lem:preparation} 
(that guarantees that $G^*(e^{t^2/2}) e^{-t^2/2} \leq 0$ for all $t \in (0,x_o)$) the following inequality holds
\begin{align*}
& \int_0^\infty \exp \left\{ \left( \int_0^x G(|f'(x)|) e^{-\frac{x^2}{2}} dx
+
\int_0^x 
\left[ G^*(e^{t^2/2}) e^{-t^2/2}  \right]dt  \right)^\frac{2\beta}{\beta+2} - \frac{x^2}{2}  \right\} dx \\
& \leq 
e^{M^\frac{2\beta}{\beta+2}} \int_0^{x_o} e^{ - \frac{x^2}{2}} dx  +
\int_{x_o}^\infty \exp \left\{ \left( M
+
\int_{x_o}^x \left[  G^*(e^{t^2/2}) e^{-t^2/2} \right]dt  \right)^\frac{2\beta}{\beta+2} - \frac{x^2}{2}  \right\} dx .
\end{align*}
Observe that by the last part of Lemma \ref{lem:preparation}, $x_o \leq 1$ so that
$\int_0^{x_o} e^{-x^2/2}dx \leq \int_0^1 e^{-x^2/2}dx \leq 1$. Combining these observations
and the fact that $G^*(e^{t^2/2}) e^{-t^2/2} \leq W^{-1}(t^2/2)$ (note that $W^{-1} \geq 0$), one has
\begin{align*}
& \int_0^\infty e^{|f(x)|^\frac{2\beta}{\beta+2}} e^{-\frac{x^2}{2}} dx  
\leq  
e^{M^\frac{2\beta}{\beta+2}} 
+ 
\int_{x_o}^\infty \exp \left\{ \left( M
+
\int_{x_o}^x  W^{-1}(t^2/2) dt  \right)^\frac{2\beta}{\beta+2} - \frac{x^2}{2}  \right\} dx .
\end{align*}
Since $0 \leq \frac{2\beta}{\beta+2} \leq 1$, 
\begin{align*}
\left( M
+
\int_{x_o}^x  W^{-1}(t^2/2) dt  \right)^\frac{2\beta}{\beta+2}    
& \leq 
M^\frac{2\beta}{\beta+2} 
+
\left( \int_{x_o}^x  W^{-1}(t^2/2) dt  \right)^\frac{2\beta}{\beta+2}  .
\end{align*}
Therefore, as an intermediate result, it holds that
$$
\int_0^\infty e^{|f(x)|^\frac{2\beta}{\beta+2}} e^{-\frac{x^2}{2}} dx  
\leq  
e^{M^\frac{2\beta}{\beta+2}}
\left(1 + \int_{x_o}^\infty \exp \left\{ 
\left(  \int_{x_o}^x  W^{-1}(t^2/2) dt  \right)^\frac{2\beta}{\beta+2}  - \frac{x^2}{2}  \right\} dx  
\right) .
$$
Our next aim is to analyse the term
$\int_{x_o}^x  W^{-1}(t^2/2) dt$ for $x \geq x_o$ in order to prove that the integral factor is finite.
By Lemma \ref{lem:preparation} and a change of variables, we have
\begin{align*} 
\int_{x_o}^x  W^{-1}(t^2/2) dt
& \leq 
\frac{1}{\kappa}\int_{\sqrt{2/3}}^x \left( \frac{t^2}{2} - \frac{\beta-1}{\beta} \log \left( \frac{t^2}{2} \right) + 1 \right)^\frac{1}{\beta} dt \\
& =
\frac{1}{\kappa}\int_{\frac{1}{3}}^{\frac{x^2}{2}} \left( u - \frac{\beta-1}{\beta} \log u + 1 \right)^\frac{1}{\beta} \frac{du}{\sqrt{2u}}. 
\end{align*}
Using the concavity property  
\begin{equation} \label{eq:concavity}
(1+y)^\gamma \leq 1 + \gamma y, \quad y > -1, \quad \gamma \in [0,1]
\end{equation}
 with $\gamma= \frac{1}{\beta}$, it follows that (for $u \geq 1/3$)
\begin{align*}
\left( u - \frac{\beta-1}{\beta} \log u + 1 \right)^\frac{1}{\beta} 
& =
u^\frac{1}{\beta} \left( 1 - \frac{\beta-1}{\beta} \frac{\log u}{u} + \frac{1}{u} \right)^\frac{1}{\beta} \\
&\leq 
u^\frac{1}{\beta} \left( 1 + \frac{1}{\beta} \left(- \frac{\beta-1}{\beta} \frac{\log u}{u} + \frac{1}{u} \right)\right)      \\
& =
u^\frac{1}{\beta} - \frac{\beta-1}{\beta^2} u^{-\frac{\beta-1}{\beta}} \log u + \frac{1}{\beta} u^{-\frac{\beta-1}{\beta}} .
\end{align*}
As a consequence, by direct computation and some algebra, the following holds
\begin{align*}
& \int_{\frac{1}{3}}^{\frac{x^2}{2}} \left[ u - \frac{\beta-1}{\beta} \log u + 1 \right]^\frac{1}{\beta} \frac{du}{\sqrt{2u}} \\
&  \leq 
\frac{1}{\sqrt{2}}
\int_{\frac{1}{3}}^{\frac{x^2}{2}} \left[
u^\frac{2-\beta}{2\beta} - \frac{\beta-1}{\beta^2} u^{-\frac{3\beta-2}{2\beta}} \log u + \frac{1}{\beta} u^{-\frac{3\beta-2}{2\beta}}
\right] du \\
& =
\frac{1}{\sqrt{2}}
\left[ \frac{2\beta}{\beta+2} u^\frac{2+\beta}{2\beta} 
- \frac{\beta-1}{\beta^2}\left(
\frac{2\beta}{2-\beta} u^\frac{2-\beta}{2\beta} \log u
- \frac{4\beta^2}{(2-\beta)^2} u^\frac{2-\beta}{2\beta} \right)
+
\frac{2}{2-\beta} u^\frac{2-\beta}{2\beta} 
\right]_{\frac{1}{3}}^{\frac{x^2}{2}} \\
& =
\kappa \left( \frac{x^2}{2}\right)^\frac{2+\beta}{2\beta}
- \frac{\sqrt{2}(\beta-1)}{\beta(2-\beta)}\left( \frac{x^2}{2}
\right)^\frac{2-\beta}{2\beta} \log \left( \frac{x^2}{2} \right)
+\frac{\sqrt{2}\beta}{(2-\beta)^2} \left( \frac{x^2}{2}
\right)^\frac{2-\beta}{2\beta} 
- c_\beta 
\end{align*}
where in the last line we used the fact that $\kappa=\frac{\sqrt{2}\beta}{\beta+2}$, and we set
$$
c_\beta \coloneqq 
\frac{\kappa}{3^\frac{2+\beta}{2\beta}}
+ \frac{\sqrt{2}(\beta-1)}{\beta(2-\beta)} \frac{\log 3}{3^\frac{2-\beta}{2\beta}} 
+\frac{\sqrt{2}\beta}{(2-\beta)^2 3^\frac{2-\beta}{2\beta}} >0 .
$$
Therefore, using \eqref{eq:concavity} with $\gamma = \frac{2\beta}{\beta+2} \leq 1$, we obtain
\begin{align*}
& \left(
\int_{x_o}^x  W^{-1}(t^2/2) dt  \right)^\frac{2\beta}{\beta+2} \\
& \leq 
\left(
\left( \frac{x^2}{2} \right)^\frac{2+\beta}{2\beta}
- \frac{\sqrt{2}(\beta-1)}{\kappa \beta(2-\beta)}\left( \frac{x^2}{2}
\right)^\frac{2-\beta}{2\beta} \log \left( \frac{x^2}{2} \right)
+\frac{\sqrt{2}\beta}{\kappa (2-\beta)^2} \left( \frac{x^2}{2}
\right)^\frac{2-\beta}{2\beta} 
\right)^\frac{2\beta}{\beta+2} \\
& =
\frac{x^2}{2}
\left(
1  
- \frac{\sqrt{2}(\beta-1)}{\kappa \beta(2-\beta)}
\frac{\log \left( x^2/2 \right)}{x^2/2}
+\frac{\sqrt{2} \beta }{\kappa (2-\beta)^2} \frac{1}{x^2/2}
\right)^\frac{2\beta}{\beta+2} \\
& \leq 
\frac{x^2}{2}
\left(
1 + \frac{2\beta}{\beta+2} \left( 
- \frac{\sqrt{2}(\beta-1)}{\kappa \beta(2-\beta)}
\frac{\log \left( x^2/2 \right)}{x^2/2}
+\frac{\sqrt{2}\beta}{\kappa (2-\beta)^2} \frac{1}{x^2/2}
\right) \right) \\
& =
\frac{x^2}{2} 
- \frac{2\sqrt{2}(\beta-1)}{\kappa (\beta+2)(2-\beta)}
\log \left( \frac{x^2}{2} \right)
+\frac{2\sqrt{2}\beta^2}{\kappa(\beta+2) (2-\beta)^2} .
\end{align*}
At this step, using the expression of $\kappa = \frac{\sqrt{2}\beta}{\beta+2}$ and $x_o \geq \sqrt{2/3}$ (from Lemma \ref{lem:preparation}), we can conclude that
\begin{align*}
& \int_{x_o}^\infty \exp \left\{ \left( 
\int_{x_o}^x  W^{-1}(t^2/2) dt  \right)^\frac{2\beta}{\beta+2} - \frac{x^2}{2}  \right\} dx \\
& \leq
e^\frac{2\beta}{(2-\beta)^2} \int_{\sqrt{2/3}}^\infty 
\left( \frac{2}{x^2}\right)^\frac{2(\beta-1)}{\beta(2-\beta)}
dx \\
& =
e^{\frac{2\beta}{(2-\beta)^2}}2^{\frac{2(\beta-1)}{\beta(2-\beta)}} \int_{\sqrt{2/3}}^\infty 
\left( \frac{1}{x}\right)^\frac{4(\beta-1)}{\beta(2-\beta)}
dx \\
& =
e^{\frac{2\beta}{(2-\beta)^2}}2^{\frac{2(\beta-1)}{\beta(2-\beta)}}
\frac{\beta(2-\beta)}{\beta^2+2\beta-4}
\left( \frac{3}{2}\right)^\frac{\beta^2+2\beta-4}{2\beta(2-\beta)} \\
& =
\frac{\sqrt{2}\beta(2-\beta)}{\beta^2+2\beta-4}
\exp \left\{ \frac{2\beta}{(2-\beta)^2} +
\frac{\beta^2+2\beta-4}{2\beta(2-\beta)} \log 3 \right\} 
\end{align*}
where we used the fact that, for $\beta > \sqrt{5}-1$, $\frac{4(\beta-1)}{\beta(2-\beta)}>1$ to ensure the convergence of the integral.
Now  it is not difficult to see that
$$
\frac{\sqrt{2}\beta(2-\beta)}{\beta^2+2\beta-4} 
\leq 
\frac{8 \sqrt{2} - 16 \sqrt{2/5}}{4(\beta-\sqrt{5}+1)}
\leq 
\frac{1}{2(\beta-\sqrt{5}+1)}
\qquad \beta \in (\sqrt{5}-1 , 2)
$$
(the maximum is reached at $\beta=\sqrt{5}-1$), 
and 
$$
\frac{\beta^2+2\beta-4}{2\beta(2-\beta)} \log 3
\leq 
\frac{1}{(2-\beta)^2} .
$$
All together, we obtained for any $\beta \in (\sqrt{5}-1,2)$
\begin{align*}
 \int_0^\infty e^{|f(x)|^\frac{2\beta}{\beta+2}} e^{-\frac{x^2}{2}} dx  
 & \leq  
 e^{M^\frac{2\beta}{\beta+2}} \left(
 1 +
\frac{1}{2(\beta - \sqrt{5}+1)}
\exp \left\{ \frac{2\beta+1}{(2-\beta)^2} \right\} \right) \\
& \leq 
\frac{1}{\beta - \sqrt{5}+1}
\exp \left\{ \frac{5}{(2-\beta)^2} \right\}  e^{M^\frac{2\beta}{\beta+2}}
\end{align*}
This leads to the desired conclusion.
\end{proof}

\begin{lem} \label{lem:preparation}
Let $V(x) := (\kappa x)^\beta$,  $x \geq 0$, $\beta \in (1,2)$, with $\kappa:= \frac{\sqrt{2}\beta}{\beta+2}$ and 
$W(x):= (\kappa x)^\beta + \log (\beta \kappa^\beta x^{\beta-1})$, $x >0$. Then,  for all $x > 0$ the following inequality holds
\begin{equation*}
W^{-1}(x) \leq \frac{1}{\kappa}\left( x - \frac{\beta-1}{\beta}\log x +  1 \right)^\frac{1}{\beta} .
\end{equation*}
Furthermore, 
$$
x \leq W \left( \frac{1}{\kappa \beta^\frac{1}{\beta}} \right) 
\quad \Longleftrightarrow \quad 
W^{-1}(x) - \frac{1}{V'(W^{-1}(x))} \leq 0 .
$$
Finally, for all $ \beta \in(\sqrt{5}-1,2)$, it holds $\frac{1}{2} \geq W\left( \frac{1}{\kappa \beta^\frac{1}{\beta}} \right)  \geq 1/3$.
\end{lem}

\begin{proof}
By studying the map $x \mapsto x - \frac{\beta-1}{\beta} \log x + (2-\beta)^2$ on $(0,\infty)$ we observe that it is decreasing and then increasing, its  minimum is achieved at $x_o= \frac{\beta-1}{\beta}$ so that $x - \frac{\beta-1}{\beta} \log x + (2-\beta)^2 >0$ for all $x>0$. Therefore, we can compute 
\begin{align*}
& W \left(\frac{1}{\kappa}\left( x - \frac{\beta-1}{\beta}\log x + 1 \right)^\frac{1}{\beta} \right) \\
& =
x - \frac{\beta-1}{\beta}\log x + 1 + \log(\beta \kappa) 
 +
\frac{\beta-1}{\beta} \log \left( x - \frac{\beta-1}{\beta}\log x + 1 \right) \\
& =
x + \frac{\beta-1}{\beta} \log \left(1 - \frac{\beta-1}{\beta} \frac{\log x}{x} + \frac{1}{x}\right) +  \log(e\beta k) .
\end{align*}
Now the map $x \mapsto - \frac{\beta-1}{\beta} \frac{\log x}{x} + \frac{1}{x}$ is decreasing and then increasing, with a minimum achieved at $e^{\frac{2\beta-1}{\beta-1}}$ so that, for any $x >0$
\begin{align*}
W \left(\frac{1}{\kappa}\left( x - \frac{\beta-1}{\beta}\log x + 1 \right)^\frac{1}{\beta} \right) 
& \geq 
x + \frac{\beta-1}{\beta} \log \left( 1 - \frac{\beta-1}{\beta}e^{-\frac{2\beta-1}{\beta-1}} \right)
+\log(e \beta k) \\
& \geq 
x + \frac{\beta-1}{\beta} \log \left( 1 - \frac{\beta-1}{\beta e^2} \right)
+ \log(e \beta k) \\
& \geq 
x + \frac{1}{2} \log \left( 1 - \frac{1}{2 e^2} \right)
+\log(e \beta k)
\end{align*}
where we used that $\frac{2\beta-1}{\beta-1} \geq 2$ and, in the last line,  that $y \mapsto y \log(1-y)$ is decreasing on $(0,1)$. Now it is clear that
$$
 \frac{1}{2} \log \left( 1 - \frac{1}{2 e^2} \right)
+\log(e \beta k) = \log \left( \frac{\sqrt{2e^2-1}\beta^2}{\beta+2}\right) >0
$$
for any $\beta \in (1,2)$ so that $W \left(\frac{1}{\kappa}\left( x - \frac{\beta-1}{\beta}\log x + 1 \right)^\frac{1}{\beta} \right) 
 \geq x$, $x >0$. Since $W$ is increasing, the first expected result follows.

For the second part of the lemma, set $y=W^{-1}(x)$ and observe that
$y - \frac{1}{V'(y)} \leq 0$ if and only if
$y V'(y)=\beta \kappa^\beta y^\beta \leq 1$, proving the desired result by monotonicity of $W$.

For the last conclusion of the lemma, note that 
$W\left( \frac{1}{\kappa \beta^\frac{1}{\beta}} \right)  =
\frac{\log(e\beta)}{\beta} + \log( \kappa )$. In particular, the map $\beta \in (1,2) \mapsto W\left( \frac{1}{\kappa \beta^\frac{1}{\beta}} \right)$ is increasing so that, for $\beta \in (\sqrt{5}-1,2)$, it holds
$$
\frac{1}{2} = \frac{\log(2e)}{2} + \log \left( \frac{\sqrt{2}}{2}\right) \geq W\left( \frac{1}{\kappa \beta^\frac{1}{\beta}} \right)
\geq 
\frac{\log(e(\sqrt{5}-1))}{\sqrt{5}-1} + \log \left(
\frac{\sqrt{2}(\sqrt{5}-1)}{\sqrt{5}+1}\right)
\simeq 0.36 \geq \frac{1}{3} .
$$
\end{proof}


\newpage 


\section{Discrete setting}

Consider a graph $(G,V)$ with vertex set $V$. Given a probability measure $\mu$ on $G$ and an operator $L=(L(x,y))_{x,y \in V}$ symmetric in $\mathbb{L}^2(\mu)$ and satisfying $L(x,y) \geq 0$ for all $x \neq y$ and $\sum_y L(x,y)=0$ for all $x$, we are interested in the following modified log-Sobolev inequality
\begin{equation} \label{eq:modified-log-sob}
\Ent_\mu (e^f) \leq c \sum_{x,y \in V} \mu(x)L(x,y) \left( e^{f(y)} -e^{f(x)} \right) (f(y)-f(x))    .
\end{equation}
The discrete gradient 
$\nabla e^f \nabla f$ reduces to $e^f |\nabla f|^2$ in the continuous setting. Therefore Inequality \eqref{eq:modified-log-sob} is equivalent to the usual log-Sobolev inequality in the continuous. However, in the discrete setting, the behavior of the classical log-Sobolev inequality, that reads
$$
\Ent_\mu (e^f) \leq \frac{c}{2} \sum_{x,y \in V} \mu(x)L(x,y) \left( e^{f(y)} -e^{f(x)} \right)^2 
$$
might be very different from the modified one. For instance, the Poisson measure on the integers, associated to a proper birth and death process, called $M/M/\infty$ in the literature, is known to satisfy the modified log-Sobolev inequality \eqref{eq:modified-log-sob} with optimal constant the parameter of the Poisson measure, while the log-Sobolev inequality  above does not hold.

There is a large activity on this topic, in relation with models coming from statistical mechanics, transport theory and some notion of Ricci curvature on graphs. To give a complete list of the literature is out of reach. Let us mention \cite{bobkov-tetali,GRST} for two papers involving some of the authors, in this direction.

We may use the following classical notations. Given $f$, $Lf(x)=\sum_y L(x,y)f(y)$. Also, by reversibility, it is easy to see that
$$
\int f (-Lg) d\mu = \frac{1}{2}\sum_{x,y} \mu(x,y)L(x,y) (g(y)-g(x))(f(y)-f(x)) .
$$
Therefore the modified log-Sobolev inequality can equivalently be written as
$$
\Ent_\mu (e^f) \leq c \int e^f (-Lf) d\mu .
$$

\begin{thm}
Assume that, on a graph $(G,V)$, the probability measure and operator $L$ as above are satisfying the modified log-Sobolev Inequality \eqref{eq:modified-log-sob} with constant $c \in (0,\infty)$.
Then, for any $\alpha > c$, and any $f$ with $\int fd\mu=0$ the following inequality holds
$$
\log \int e^{f} d\mu 
\leq  \frac{c}{\alpha -c}  \int e^{\alpha  |Lf|}d\mu .
$$
\end{thm}

\begin{proof}
Fix $\lambda \in (0,1]$ and a function $f$ with mean zero with respect to $\mu$.
By the variational formula of the entropy
$\Ent_\mu(h)=\sup\{\int hg d\mu : \log \int e^g d\mu \leq 1\}$, applied with $g= -\alpha \lambda Lf - \log \int e^{-\alpha \lambda Lf}d\mu$,
together with the modified log-Sobolev inequality we have
$$
\alpha \lambda \int e^{\lambda f}(-Lf) d\mu - \int e^{\lambda f} d\mu \log  \int e^{-\alpha \lambda Lf}d\mu
\leq \Ent_\mu(f^{\lambda f})
\leq c \int e^{\lambda f} (-L (\lambda f)) d\mu 
$$
Therefore, for $\alpha > c$, it holds
$$
\int e^{\lambda f} (-L (\lambda f)) d\mu \leq \frac{1}{\alpha -c}  \int e^{\lambda f} d\mu \log  \int e^{-\alpha \lambda Lf}d\mu .
$$
Applying the modified log-Sobolev inequality we obtain
$$
\Ent_\mu(e^{\lambda f}) \leq c \int e^{\lambda f} (-L (\lambda f)) d\mu
\leq \frac{c}{\alpha -c}  \int e^{\lambda f} d\mu \log  \int e^{-\alpha \lambda Lf}d\mu
$$
The rest of the proof goes as for the proof of Theorem \ref{thm:CMP-Lf}. Details are left to the reader.
\end{proof}

Finally we specialized to the Poisson measure $\pi(k)=e^{-\lambda} \frac{\lambda^k}{k!}$, on the integers $k=0,1 \dots$ 
and the $M/M/\infty$ queuing process. We will use our convexity argument  to obtain some new exponential inequalities.

Set $\nabla f(k) \coloneqq f(k+1)-f(k)$, $k=0,1\dots$ for the discrete gradient of a function $f$ on the integers.

\begin{thm}
Let $\lambda >0$ and for $k$ integers, define $\pi(k)=e^{-\lambda} \frac{\lambda^k}{k!}$. 
Define 
$$
G(x) \coloneqq x \mathds{1}_{[0,1]} + \exp\left\{\lambda \left(x + \left(1 + \frac{2}{x} \right) \log (\lambda x)  \right) e^{x} \right\} \mathds{1}_{(1,\infty)}, \qquad x  > 0,
$$
Then, there exists a constant $c$ (that depends only on $\lambda$) such that for any $f \colon \mathbb{N} \to \mathbb{R}$ with $f(0)=0$, it holds that
$$
\log \left( \int e^f d\pi  \right) 
\leq   
c + \int G(|\nabla f|) d\pi  
$$
and a constant $d$ (that depends only on $\lambda$) such that for $f$ with $\pi$-mean zero, \textit{i.e.}\ $\int f d\pi=0$, the following holds
$$
\log \left( \int e^f d\pi  \right) 
\leq  
c+d \int G(|\nabla f|) d\pi .
$$
\end{thm}

\begin{remark}
The above integrals are understood on $\mathbb{N}$, \textit{i.e.}\ $\int g d\pi = \sum_{n=0}^\infty g(n) \pi(n)$.   

The form of $G$ is devised to ensure that $x \leq C G(x)$ for some constant $C>0$. In fact, when $x$ tends to $0$,  
$\exp\left\{\lambda \left(x + \left(1 + \frac{2}{x} \right) \log (\lambda x)  \right) e^{x} \right\}$ is much smaller that $x$.
\end{remark}

\begin{proof}
Let $f \colon \mathbb{R} \to \mathbb{R}$ with $f(0)=0$.
By duality, we have
\begin{align*}
|f(n)|
& = 
|\sum_{k=0}^{n-1} \nabla f(k) | \\
& \leq 
\sum_{k=0}^{n-1} |\nabla f(k)| \frac{1}{\pi(k)} \pi(k) \\ 
& \leq 
\sum_{k=0}^{n-1} G(|\nabla f(k)|)\pi(k) + \sum_{k=0}^{n-1} G^* \left(\frac{1}{\pi(k)} \right) \pi(k)
\end{align*}
where we set: \\
$G(x) \coloneqq x \mathds{1}_{[0,1]} + \exp\left\{\lambda \left(x + \left(1 + \frac{2}{x} \right) \log (\lambda x)  \right) e^{x} \right\} \mathds{1}_{(1,\infty)}$ and $G^*(y)=\sup_{x > 0}\{xy-G(x)\}$, for all $y \geq 0$. Therefore, 
\begin{align*}
\log \left( \int e^f d\pi  \right) 
& \leq 
 \sum_{n=0}^\infty G(|\nabla f(n)|) \pi(n) +
 \log \left( \pi(0) + \sum_{n=1}^\infty \exp \left\{ 
 \sum_{k=0}^{n-1} G^* \left(\frac{1}{\pi(k)} \right) \pi(k)
 \right\} \pi(n) \right)
\end{align*}
and we are left with proving that the second term in the right hand side of the latter is bounded. In fact we need to prove that the sum is convergent. By \eqref{eq:G*}, we have, for $n$ large enough
\begin{align*}
\sum_{k=0}^{n-1} G^* \left(\frac{1}{\pi(k)} \right) \pi(k)
& \leq 
\sum_{k=0}^{k_o} G^* \left(\frac{1}{\pi(k)} \right) \pi(k) + \mathds{1}_{n \geq k_o} \sum_{k=k_o}^{n-1} \log k 
  - \log \lambda  
  - \frac{1}{k}
  + \frac{3}{k \log k}  \\
& \leq 
C + (\log (n-1)! - n \log \lambda  - \log n + D\log_2 n)\mathds{1}_{n \geq k_o} 
\end{align*}
for some constant $C$ that depends only on $\lambda$ and some universal constant $D$. Above $k_o$ is determined by Lemma \ref{lem:G*}.
As a consequence, considering a bigger constant $C$ if needed,
\begin{align*}
    \sum_{n=1}^\infty \exp \left\{ 
 \sum_{k=0}^{n-1} G^* \left(\frac{1}{\pi(k)} \right) \pi(k)
 \right\} \pi(n)
 & \leq 
e^C  \left( 1 + \sum_{n=2}^\infty   (n-1)! \lambda^{-n} \frac{(\log n)^D}{n} \pi(n) \right) \\
& =
e^{C} \left( 1 + e^{-\lambda} \sum_{n=2}^\infty \frac{(\log n)^D}{n^2}  \right) \\
& < \infty .
\end{align*}
This proves the first part of the theorem.

Consider a function $f$ with mean $0$ with respect to the Poisson measure, \textit{i.e.}\ such that 
$\pi(f) \coloneqq \sum_{n=0}^\infty f(n)\pi(n)=0$. Then, 
\begin{align} \label{eq:start-poisson}
\log \left( \sum_{n=0}^\infty e^{|f(n)|}\pi(n) \right) 
& = 
\log \left( \sum_{n=0}^\infty e^{|f(n) - f(0) + f(0) - \pi(f)|}\pi(n) \right) \nonumber \\
& \leq 
\log \left( \sum_{n=0}^\infty e^{|f(n) - f(0)|}\pi(n) \right) 
 + |f(0)-\pi(f)| .   
\end{align}
The first term in the right hand side can be bounded using the first part of the theorem. We need to bound the second term 
$|f(0)-\pi(f)| \leq \sum_{n=0}^\infty |f(n)-f(0)|\pi(n)$. To that purpose we may use a direct computation. Indeed, we have
\begin{align*}
\sum_{n=0}^\infty |f(n)-f(0)|\pi(n)
& =
\sum_{n=1}^\infty |\sum_{k=0}^{n-1} \nabla f(k)|\pi(n) \\
& \leq 
\sum_{k=0}^\infty |\nabla f(k)| \pi(k) \sum_{n =k+1}^\infty \frac{\pi(n)}{\pi(k)} \\
& \leq 
A \sum_{k=0}^\infty |\nabla f(k)| \pi(k)
\end{align*}
where 
$$
A\coloneqq \sup_{k=0}^\infty \frac{1}{\pi(k)} \sum_{n =k+1}^\infty \pi(n) .
$$
The above inequality belongs to the family of Hardy's inequality\footnote{After 
\cite{muckenhoupt}, Miclo \cite{miclo} developed Hardy-type inequalities in the discrete setting 
related to functional inequalities, see also \cite{bobkov-gotze} and \cite{barthe-roberto-03,BCR}.}. It should be noticed that the constant $A$ is best possible as one can convince oneself by considering the test function $f=\mathds{1}_{[k,\infty)}$. It is easy to check that $A$, that depends only on $\lambda$, is finite.
 From \eqref{eq:start-poisson} we have
$$
\log \int e^f d\pi \leq c+ \int G(|\nabla f|) d\pi + A \int |\nabla f| d\pi 
\leq c + d \int G(|\nabla f|) d\pi .
$$
where in the last inequality we used the fact that $x \leq C G(x)$, for some constant $C$ that depends only on $\lambda$, and all $x \geq 0$.
\end{proof}

In the proof above we used some technical lemmas. Recall the following well-known non-asymptotic version of the Stirling formula
(see \textit{e.g.}\ \cite{robbins})
\begin{equation} \label{eq:stirling}
e^{\frac{1}{1+ 12n}} \leq \frac{n!}{n^n e^{-n} \sqrt{2\pi n}} \leq    e^{\frac{1}{12n}} , \qquad n \geq 1 
\end{equation}
For simplicity we denote $\log_2=\log(\log)$ and $\log_3 = \log_2(\log)$ for the iterated logarithms. Given a function $F$ we denote $F^{-1}$ its inverse function, when it exists.

\begin{lem}\label{lem:phi}
 Define $\Phi_o(x)=x\log x$ for $x \geq e$, and $\Phi_1(x) =  (1 + \frac{1}{\log_2 x}) x \log x$ for $x \geq e^2$. Then, $\Phi_o$ and $\Phi_1$ are one to one increasing on $[e,\infty)$ onto itself and on $[e^2,\infty)$ onto 
 $[4e^2,\infty)$ respectively, their inverse functions $\Phi_o^{-1}$ and $\Phi_1^{-1}$ satisfy, for any $x \geq e$
 (respectively any $x \geq e^2$)
 $$
 \frac{x}{\log x} \leq \Phi_o^{-1}(x) 
 \qquad \mbox{and} 
 \qquad 
 \Phi_1^{-1}(x) \leq  \frac{x}{\log x} \left( 1 +  \frac{1}{\log x} \right) .
 $$
\end{lem}

\begin{proof}
The maps $\Phi_o, \Phi_1$  are clearly increasing therefore one to one.
 For the lower bound observe that, for $x \geq e$,
 $$
 \Phi_o \left( \frac{x}{\log x} \right)
 =
 x\frac{\log x - \log_2 x}{\log x} 
 \leq x 
 $$
 while for the upper bound  (using that $\log(1+\frac{1}{\log x}) \geq 0$ and that $-\log_2 x + \log(1+\frac{1}{\log x} \leq 0$)
 we have for $x \geq e^2$
 \begin{align*}
\Phi_1 \left( \frac{x}{\log x}\left(1 + \frac{1}{\log x} \right) \right)
 & =
 x \left(1+ \frac{1}{\log x} \right) 
 \frac{\log x - \log_2 x + \log[1+ \frac{1}{\log x }]}{\log x} \times \\
& \phantom{AAAAAAAA} \times\left( 1 + \frac{1}{\log \left( \log x - \log_2 x + \log (1+\frac{1}{\log x}) \right)}  \right)
 \\
& \geq 
x \left(1+ \frac{1}{\log x} \right) 
\left( 1 - \frac{\log_2 x}{\log x} \right)
\left( 1 + \frac{1}{\log_2 x}  \right)  .
 \end{align*}
Now it is not difficult to check that the product of the three brackets in the right hand side of the latter is greater than or equal to 1 for any $x \geq e^2$, proving that 
 $\Phi_1^{-1}(x) \leq \frac{x}{\log x}\left(1+ \frac{1}{\log x} \right)$ as announced.
\end{proof}

\begin{lem} \label{lem:G*}
For $\lambda >0$ define 
$$
G(x) \coloneqq x \mathds{1}_{[0,1]} + \exp\left\{\lambda \left(x + \left(1 + \frac{2}{x} \right) \log (\lambda x)  \right) e^{x} \right\} \mathds{1}_{(1,\infty)}, \qquad x  > 0,
$$
and $\psi(x) \coloneqq \log_2 G(x)$, $x \geq \max(1,1/\lambda)$.
Then there exists $x_o$ (that depends only on $\lambda$) such that for any $x \geq x_o$, the following holds
$$
\psi^{-1}(x) \leq x - \log (\lambda x)   
.
$$
Furthermore,
$$
G^*(x) \coloneqq \sup_{y > 0} \{xy - G(y) \}, \qquad x \geq 0 
$$
satisfies the following inequality
$$
G^*(x) \leq x \psi^{-1}(\log_2 \Phi_1^{-1}(x)) 
,  \qquad x \geq x_o
$$
where 
$\Phi_1^{-1}$ is the inverse function of $\Phi_1 \colon x \mapsto x\log x(1+\frac{1}{\log_2 x})$ studied in Lemma \ref{lem:phi}. In particular, for any integer $k$ large enough (how large depends only on $\lambda$), $\pi(k)\coloneqq \frac{e^{-\lambda} \lambda^k}{k!}$ satisfies
\begin{equation} \label{eq:G*}
\pi(k)G^*(1/\pi(k)) 
\leq 
\log k 
  - \log \lambda  
  - \frac{1}{k}
  + \frac{3}{k \log k} .
\end{equation}
\end{lem}

\begin{proof}
Note that $\psi$ is increasing and therefore one to one from $(\max(1,1/\lambda) ,\infty)$ onto its image. Its inverse function is well defined and we have, for $x$ large enough
\begin{align*}
\psi \left(x -  \log (\lambda x) 
\right)
& =
x 
+  \log \left( \frac{1}{x} \left[ 
x - \log (\lambda x) + \left(1 + \frac{2}{x - \log(\lambda x)} \right)\log ( \lambda(x-\log(\lambda x)))  \right]\right)  \\
& =
x 
+  \log \left( 1 + \frac{1}{x} \left[ \log(1 - \frac{\log (\lambda x)}{x}) + \frac{2}{x - \log(\lambda x)}  \log  ( \lambda(x-\log(\lambda x)))  \right] \right) 
 \geq x .
\end{align*}
The first statement on $\psi^{-1}$ follows. 

On $(1,\infty)$, 
$$
G'(x)
=
G(x)\log G(x) \left(1+\frac{2+x+x^2-\log(\lambda x)}{x(x^2 + (2+x) \log(\lambda x))}\right)
$$
so that $G'$ is increasing for $x \geq \max(1,1/\lambda)$ with inverse function we denote by ${G'}^{-1}$. In particular, for $x$ large enough,
$$
G^*(x)=x{G'}^{-1}(x) - G({G'}^{-1}(x)) \leq  x{G'}^{-1}(x) .
$$
Moreover, for any $x$ large enough, 
$$
G'(x) \geq G(x) \log G(x) (1+\frac{1}{\log_2 G(x)})= \Phi_1(G(x))
$$
with $\Phi_1(x):=x\log x(1+\frac{1}{\log_2 x})$. Therefore, recall that $\psi := \log_2 G$,  for $x$ large enough it holds
$$
{G'}^{-1}(x) \leq G^{-1} (\Phi_1^{-1}(x)) = \psi^{-1} ( \log_2 \Phi_1^{-1}(x)) 
$$
leading to the announced upper bound on  $G^*$. 

Now by Lemma \ref{lem:phi}, we deduce that, for $x$ small enough, setting $y=\log(1/x)$,
\begin{align*}
xG^*(1/x) 
& \leq 
\psi^{-1}(\log_2 \Phi^{-1}((1/x)))
 \\
 & \leq 
 \psi^{-1} \left( \log_2 \left( \frac{1}{x\log (1/x)} \left[1 + \frac{1}{\log(1/x)} \right] \right) \right)
\\
& = 
\psi^{-1} \left( \log \left[y - \log y + \log\left( 1 + \frac{1}{y} \right) \right] \right) .
\end{align*}
Since $\log(1+z) \leq z$, it holds
\begin{align*}
\log \left[y - \log y + \log\left( 1 + \frac{1}{y} \right) \right]
\leq 
\log y - \frac{\log y}{y} + \frac{1}{y^2}
.
\end{align*}
Therefore, using the upper bound on $\psi^{-1}$, in the large we obtain that 
\begin{align*}
xG^*(1/x)  
& \leq 
\psi^{-1} \left( \log y - \frac{\log y}{y} + \frac{1}{y^2} \right)  \\
& \leq 
\log y - \frac{\log y}{y}  + \frac{1}{y^2} - \log \lambda - \log \left( \log y - \frac{\log y}{y} + \frac{1}{y^2} \right) 
\\
& \leq 
\log y - \log_2 y - \log \lambda  - \frac{\log y}{y} + \frac{2}{y} .
\end{align*}
It follows that (recall that $y=\log(1/x)$), for $k$ large enough,
\begin{align*}
 \pi(k)G^*(1/\pi(k)) 
 & \leq
 \log_2 \frac{1}{\pi(k)} - \log_3  \frac{1}{\pi(k)} - \log \lambda 
  -  \frac{\log_2 \frac{1}{\pi(k)}}{\log \frac{1}{\pi(k)}}
 + \frac{2}{ \log (1/\pi(k))} 
 .
\end{align*}
By the approximation of the Stirling formula \eqref{eq:stirling}, for $k$ large enough
\begin{align*}
\log \frac{1}{\pi(k)} 
& \leq 
k \log k - k \log \lambda - k + \frac{1}{2}\log k + \lambda + \frac{1}{2} \log (2\pi) + \frac{1}{12 k} \\ 
& \leq 
 k \log k 
\end{align*}
and
\begin{align*}
\log \frac{1}{\pi(k)} 
& \geq 
k \log k - k \log \lambda - k + \frac{1}{2}\log k + \lambda + \frac{1}{2} \log (2\pi) \\ 
& \geq 
\frac{2}{3} k \log k  .
\end{align*}
Therefore, for $k$ large enough
\begin{align*}
\log_2 \frac{1}{\pi(k)} 
& \leq 
\log k + \log_2 k .
\end{align*}
On the other hand, using again \eqref{eq:stirling}, for $k$ large enough\begin{align*}
\log_2 \frac{1}{\pi(k)} 
& \geq 
\log \left( k \log k - k \log \lambda - k + \frac{1}{2}\log k + \lambda + \frac{1}{2} \log (2\pi)  \right) \\ 
& \geq 
\log k + \log_2  k -1 \\
& \geq 
\log k
\end{align*}
so that
$$
\log_3 \frac{1}{\pi(k)}  \geq \log_2 k .
$$
It follows that
\begin{align*}
 \pi(k)G^*(1/\pi(k)) 
 & \leq
 \log k 
  - \log \lambda  
  - \frac{1}{k}
  + \frac{3}{k \log k}
  .
\end{align*}
This is the desired thesis.
\end{proof}
    
\section*{Acknowledgements} 
The authors 
warmly thank the anonymous referees for their careful reading of the
manuscript and for pointing out to them some relevant references that substantially improved the presentation of the paper.


\bibliographystyle{plain}
\bibliography{paper-exponential-integrability.bib}

\end{document}